\documentclass[a4paper,11pt]{article}
\usepackage[left=3cm,right=3cm]{geometry}
\usepackage{amssymb}
\usepackage{amsfonts, amsmath, amsthm,amsbsy}
\usepackage[dvips,pdftex]{graphicx}
\usepackage{epstopdf}
\usepackage{hyperref}
\usepackage{mathrsfs}
\usepackage{color}
\usepackage{subcaption}

\newtheorem{theorem}{Theorem}

\newtheorem{lemma}[theorem]{Lemma}

\numberwithin{equation}{section}
\def\ve{\varepsilon}

\newcommand{\vr}{\varepsilon}
\newcommand{\ds}{\displaystyle}

\title{Numerical approximations to the scaled first derivatives  of a two parameter singularly perturbed problem}
\author{E.\ O'Riordan \footnote{School of Mathematical Sciences, Dublin City
University, Dublin 9, Ireland.\ email: eugene.oriordan@dcu.ie}\quad \and M.\ L.\ Pickett\footnote{Department of Mathematics, University of Portsmouth, U.K. \ email: maria.pickett@port.ac.uk}}

\begin{document}
\maketitle
\begin{abstract}
A singularly perturbed problem involving two singular perturbation parameters is discretized using the classical upwinded finite difference scheme on an appropriate piecewise-uniform Shishkin mesh. Scaled discrete derivatives (with scaling only used within the layers) are shown to be parameter uniformly convergent to the scaled first derivatives of the continuous solution.
\end{abstract}



\section{Introduction}

A characteristic feature of singularly perturbed problems is the appearance of steep gradients in the solution.
In order to generate pointwise accurate parameter-uniform \cite{fhmos} numerical approximations to the solution in the layer regions, where the steep gradients occur, it is useful to identify the correct scale of the gradients. In the case of singular perturbation problems involving one perturbation parameter,  this scale is normally some inverse power of the singular perturbation parameter. In the case of singular perturbation problems involving two perturbation parameters, the scale of the gradients appearing in the layer regions can   depend on one or both singular perturbation parameters. Outside the layer regions, the gradients are of order one. In this paper, we generate pointwise accurate numerical approximations to both the solution and the scaled first derivative of the solution. The first derivative of the solution is unbounded within the layers and so we estimate the accuracy of the appropriately scaled first derivative within the layered regions.

In the case of singularly perturbed boundary value problems of the form
\[
-\ve u'' +a(x) u'(x) +b(x) u = f(x), x \in (0,1); \quad a(x), b(x) >0;
\]
which contain a single perturbation parameter $0<\ve \leq 1$, parameter-uniform pointwise error bounds \cite{fhmos} on numerical approximations to the scaled first derivative $\ve u'$ have been established \cite{andreev1,andreev2, fhmos}. In these publications, a scaling factor of $\ve$ is applied (to the error in estimating $u'$) throughout the domain $[0,1]$. Kopteva and Stynes \cite{kopteva} established a first order error bound for approximations to the first derivative of the solution, where the scaling was only applied within the computational layer region, where $x_i \leq C\ve \ln N$. Shishkin \cite{shish1,shish2} examined a more sophisticated metric, which involved the scaling factor smoothly changing from a scale of $\ve$ for $x \leq \ve$  to no scaling outside the analytical layer region, where $x \geq C\ve \ln (1/\ve)$. However, Shishkin \cite{shish2} also established that a numerical method combining an upwind finite difference scheme with a piecewise-uniform layer-adapted  mesh is not a parameter-uniform numerical method in this new metric. In this paper, we will establish parameter-uniform bounds on approximations to the scaled first derivative of the solution of a two parameter singularly perturbed boundary value problem, where we simply scale (by appropriate factors) within the analytical layer regions only. Our method of proof is based on the analysis in \cite{gracia1,gracia2,gracia3}, which dealt with singularly perturbed parabolic and elliptic problems containing a single perturbation parameter.

In \cite{highorder2parameter} a second order parameter-uniform scheme was constructed for the two parameter problem considered below. Using the same scaling (as in the current paper) such a scheme automatically has essentially first order convergence for the scaled first derivatives. However, the finite difference operator involved in the scheme from  \cite{highorder2parameter} is rather complicated. Here, we deal with the  simple upwind finite difference operator, which is only a first order scheme for the solution. However, this simple numerical method generates first order (up to logarithmic factors) approximations to the scaled first derivatives. The key ingredient within the numerical method is the design of a suitable piecewise-uniform Shishkin mesh.

 Note that in \cite{linss,lr}, the transition parameters for the Shishkin mesh, involve the roots of a quadratic function, which is non-trivial in the case of variable coefficients. Below the appropriate scaled weighting factors to be used in estimating the derivatives and  the transition parameters for the mesh are explicitly stated in terms of the two singular perturbation parameters $\ve$ and $\mu$.  In \cite{Priyadharshini} the authors consider numerical approximations to the scaled first derivative of the solution of the singularly perturbed two parameter problem considered in the current paper. The method of proof is based on the argument given in \cite{fhmos} for the special case of $\mu=1$.  However, many of the main results (e.g. \cite[Lemma 5]{Priyadharshini}) are stated without  proof and certain crucial steps in the supplied proofs do not hold up to scrutiny (e.g. see the bound (16) in \cite[Lemma 10]{Priyadharshini} and note that in the left layer region \cite[Lemma 7]{Priyadharshini} simply yields that the error is bounded by $CN^{-1}$.). In this paper, we use a different method of proof from \cite{fhmos} and all the relevant details for the proofs are supplied.

In the broad context of singularly perturbed problems, there are two main classes of problems (reaction-diffusion and convection-diffusion) studied in the literature. One attraction of considering the two-parameter-problem is that this problem class encompasses both of these classes. Nevertheless, in the proofs of the main results given below, we see that this classification into two types of problem classes  persists. The numerical analysis presented below re-enforces the distinction between singularly perturbed problems of reaction-diffusion type and those of  convection-diffusion type.

The paper is structured as follows. In Chapter 2, {\it a priori} bounds on the first five derivatives of the continuous solution are established. These bounds motivate the scaling  used in the definition of the scaled $C^1$-norm, which is the norm used to measure the error in the numerical approximations. The numerical method is constructed in Chapter 3. Chapter 4 is the core chapter, where the nodal error analysis is given. The global error analysis is conducted in Chapter 5 and  a numerical example is given in Chapter 6. The technical details of  the  proofs  of some of the theoretical results are given in the Appendices. 

Notation: Throughout the paper, $C$ denotes a generic constant that is independent of the singular perturbation parameters $\ve, \mu$ and the number of mesh elements $N$. We adopt the following notation for the semi-norms of the solution:
\[
\vert z \vert _k := \max _{x \in [0,1]} \Bigl \vert \frac{d^k z}{dx ^k} \Bigr \vert ,\qquad  \Vert z \Vert := \max _{x \in [0,1]} \vert z(x) \vert  .
\]
 The following notation appears throughout the paper:
\[
\theta := \max \{ 1, \frac{\alpha \mu ^2}{\gamma \ve }\}, \qquad  \rho _L := \frac{1}{2}\sqrt{\frac{\gamma \alpha}{\theta \ve }}\quad \hbox{and}  \quad \rho _R:= \sqrt{\frac{\theta \gamma \alpha}{ \ve }}. 
\]
The analytical layer widths are denoted by $\tau _L, \tau _R$ and the computational layer widths are denoted by  $\sigma _L, \sigma _R$. 

\section{Continuous problem}

Find $u \in  C^5 (\Omega) \cap C^0(\bar \Omega) $ such that
\begin{subequations}\label{cp1}
\begin{eqnarray}
L_{\ve, \mu}u:= -\ve u'' +\mu a(x) u' +b(x) u = f(x),\ x \in \Omega :=(0,1), \label{diffeq} \\
u(0)=0,\quad u(1) =0,  \\
a (x) > \alpha  > 0,\ b(x) > \gamma a(x)> 0,\quad x \in \Omega .
\end{eqnarray}
\end{subequations}
 The functions $a,b$ and $f$ are assumed to be sufficiently smooth  on $\Omega $ and the perturbation parameters satisfy $0 <  \ve   \leq 1$, $0 \leq \mu  \leq 1$. Since the problem (\ref{cp1}) is linear, there is no loss in generality in assuming zero boundary conditions. Our interest lies in the case where $\ve, \mu$ are both small parameters.
 Given the constraint (\ref{cp1}c), there is no loss in generality in assuming that
 \begin{equation}\label{assum}
 b \pm  2\mu \max \{ a'\} > 0;
 \end{equation}
 as the case where $\mu \geq \mu _0 >0$, and $\mu _0$ is a fixed positive constant, has been dealt with in earlier publications \cite{gracia3}.

As in \cite{pickett} the problem naturally splits into the two separate cases of:
\[
0 \leq \frac{\alpha \mu ^2}{\gamma \ve} \leq 1 \quad {\rm and} \quad
\frac{\alpha \mu ^2}{\gamma \ve} \geq 1 .
\]
We refer to the first case as the reaction-dominated case and the second case as the convection-dominated case. We associate the following parameter
\begin{equation}
\frac{\alpha }{\gamma \ve}  \geq \theta := \max \{ 1, \frac{\alpha \mu ^2}{\gamma \ve } \} \geq 1;
\end{equation}
 with this division of the parameter space $P_{\ve,\mu}:= \{ (\ve,\mu): 0 <  \ve   \leq 1, \ 0 \leq \mu  \leq 1\}$. Our first result establishes preliminary parameter-explicit bounds on the continuous solution and it's derivatives.

\begin{lemma}\label{apriori}
\begin{subequations}\label{crude}
Assume $a,b,f \in C^3(\Omega)$, then the solution $u$ of problem (\ref{cp1}) satisfies
\begin{eqnarray}
\Vert u \Vert &\leq & 
\frac{1}{\gamma \alpha} \Vert f \Vert ;\\
 \sqrt{\ve \theta} \vert  u \vert _1 &\leq & C(1+\theta) \Vert u \Vert + C \Vert f \Vert ;\end{eqnarray}
and, for all $k$ such that  $2 \leq k \leq 5$;
\begin{equation}
\ve  ^{k/2} \vert  u \vert _k  \leq  C\theta ^{(k/2-1)} (1+\theta) \Vert u \Vert + C \sum _{j=0}^{k-2} \ve ^{j/2} \theta ^{(k-j-2)/2} \vert f \vert _j.
\end{equation}
\end{subequations}
\end{lemma}
\begin{proof}
We follow the argument in \cite[Lemma 2.2]{maria1}.
By the maximum principle $\Vert u\Vert \leq C$. Given any $x \in (0,1)$, we construct an open neighbourhood $N_x :=(p,p+r)$ such that $x \in N_x \subset (0,1)$. By the Mean Value Theorem, there exists a $y \in N_x$ such that
\[
\vert u'(y) \vert = \vert \frac{u(p+r)-u(p)}{r} \vert \leq \frac{2 \Vert u \Vert}{r}.
\]
Note that
\begin{eqnarray*}
u'(x) &=&u'(y) + \int _{t=y}^x u'' \ d t \quad
= \quad u'(y) + \frac{1}{\ve} \int _{t=y}^x \mu au'+bu -f \ d t \\
&=&u'(y) + \frac{\mu}{\ve} ((au)(x)-(au)(y))  -\frac{1}{\ve}\int _{t=y}^x \mu a'u-bu +f \ d t. \end{eqnarray*}
Thus
\[
\vert u'(x) \vert \leq C(\frac{1}{r}+ \frac{\mu}{\ve}  +\frac{r}{\ve}) \Vert u \Vert + \frac{r}{\ve} \Vert f \Vert.
\]
By taking the radius $r$ of the neighbourhood $N_x$ to be
\[
r = \sqrt{\frac{\ve \gamma}{2\theta \alpha}};
\]
we obtain the desired bound on $\vert u' \vert$. Use the differential equation (\ref{diffeq}) to obtain the bound on the second derivative, by observing that
\[
\ve \vert  u'' \vert = \vert \mu a u' -bu+f\vert \leq  C \sqrt{\theta \ve} \vert u' \vert + C (\Vert u \Vert + \Vert f \Vert).
\]
Differentiating both sides of the differential  (\ref{diffeq}) we  get that
\[
\ve \sqrt{\ve} \vert  u''' \vert  \leq  C \ve \sqrt{\theta } \vert u'' \vert + C \sqrt{\ve} (\vert u '\vert + \Vert u \Vert + \vert f' \vert).
\]
Repeating  the above  argument, we obtain the stated bounds on the third derivative. Continue this argument to obtain the bounds on all the higher derivatives.
\end{proof}
In order to obtain parameter-uniform error estimates on the numerical approximations, constructed in later sections, we decompose the solution  into regular and singular components. The regular component is constructed so that the first three derivatives of this component are bounded independently of the small parameters $\ve , \mu$.

 The continuous solution of (\ref{cp1})  is decomposed into the following sum
\begin{subequations}
\begin{equation}\label{decompf}
u(x) = v(x) + \bigl( (u-v)(0)\bigr) w_{L}(x) + \bigl((u-v-w_L)(1)\bigr) w_{R}(x)
 \end{equation}
where $w_{L}$ and $w_{R}$ satisfy homogeneous differential
equations and
\begin{eqnarray}
L_{\ve,\mu}v &=& f, \ (v(0) \textrm{ and } v(1)
\textrm{ chosen \ appropriately),}\label{eqnv1}\\
L_{\ve,\mu}w_{R} &=& 0,  \  w_R(0)=0, \ w_{R}(1) = 1,\ \label{eqnwr1}\\
 L_{\ve,\mu}w_{L} &=& 0, \
 w_{L}(0) = 1
 \label{eqnwl1}\\
&& \textrm{if}\ \mu^2 \leq \frac{\gamma\ve}{\alpha}, w_{L}(1) =
0 \ \textrm{else}\ w_{L}(1)\ \textrm{is chosen appropriately} \nonumber .
\end{eqnarray}
\end{subequations}
  We introduce the following notation for the reduced  differential operators $L_0,L_\mu $,
  \[
  L_0z:=bz \quad \hbox{and} \quad L_\mu z := \mu a z'+bz.
  \]
	
	In the next Theorem, we refine the bounds on the continuous solution $u$ given in Lemma 1. These sharper bounds identify both the location and the scale of the layers, which are used in the construction of the piecewise-uniform Shishkin mesh \cite{fhmos}.
	In addition, these bounds identify the appropriate scaling to use when estimating the error in approximating the first derivatives of the continuous solution $u$. For example, from these bounds we see that
	\begin{eqnarray*}
 \vert u'(x) \vert \leq C, \qquad \hbox{for} \qquad  2\frac{\sqrt{\ve \theta }}{\sqrt{\gamma \alpha}} \ln \frac{1}{\sqrt{\ve \theta }}\leq  x \leq 1 - \sqrt{\frac{\ve}{\gamma \alpha \theta}} \ln \sqrt{\frac{\theta}{\ve}}.
	\end{eqnarray*}
\begin{theorem}\label{boundwrightleft} Assume that $a \in C^7(\Omega), b,f \in C^9(\Omega)$.
Boundary conditions $v(0), \ v(1)$ for the regular component $v$ can be chosen so that the derivatives of the  regular component (defined in (\ref{decompf}b)) satisfy the bounds
\begin{equation}\label{boundvf21}
(i) \qquad \vert v \vert _k \leq C(1+\bigl(\sqrt{\frac{\ve}{\theta}}\bigr)^{3-k}), \quad \textrm{for}
\quad 0\leq k \leq 5.
\end{equation}
When the solution $u$ of problem (\ref{cp1}) is decomposed as in
(\ref{decompf}), the singular components $w_{L}$ and $w_{R}$ (defined in (\ref{eqnwr1}, \ref{eqnwl1}))
satisfy the following bounds
\begin{subequations}\label{exp-bound-w}
\begin{eqnarray}
&(ii)& \quad |w_{L}(x)| \leq C e^{-\frac{\sqrt{\gamma \alpha}}{2\sqrt{\ve \theta }}x} \ , \
 |w_{R}(x)| \leq C e^{-\frac{\sqrt{\gamma \alpha \theta}}{\sqrt{\ve }}(1-x)};\\
 &(iii)& \quad |w_{R}(x)|_k \leq C \bigl(\sqrt{\frac{\theta}{\ve }} \bigr)^k e^{-\frac{\sqrt{\gamma \alpha \theta}}{\sqrt{\ve }}(1-x)},\quad 1\leq k \leq 5; \\
 &(iv)& \quad |w_{L}(x)|_k \leq C\bigl(\frac{1}{\sqrt{\ve \theta }}\bigr)^k (1+\theta ^{k-3})e^{-\frac{\sqrt{\gamma \alpha}}{2\sqrt{\ve \theta }}x},\quad 1\leq k \leq 5.
\end{eqnarray}
\end{subequations}
\end{theorem}
\begin{proof} The proof is given in Appendix A. 
\end{proof}

 Based on the bounds (\ref{exp-bound-w}b) and (\ref{exp-bound-w}c), we identify  the decay rates in each of the layer regions
by \begin{equation}\label{decay-rates}
\rho _L := \max \{ 1, \frac{1}{2}\sqrt{\frac{\gamma \alpha}{\theta \ve }} \}\quad \hbox{and}  \quad \rho _R:= \max \{ 1, \sqrt{\frac{\theta \gamma \alpha}{ \ve }} \}
\end{equation}
and the associated  layer widths (for the continuous solution) to be
\[
\frac{1}{\rho _L }  \ln \rho _L   \quad \hbox{and} \quad  \frac{1}{\rho _R }  \ln \rho _R .
\]
Throughout the paper we shall assume that the parameters $\ve $ and $\mu $ are such that  $\rho _L  >1$ and  $\rho _R> 1$, as the case where $\rho _L  =1$ (or $ \rho _R= 1$)
means no layer appears on the left (or on the right) and this case can be analysed  using classical arguments. 

Note that
\[
\vert w'_L(x) \vert \leq C, \ x \geq \frac{1}{\rho _L }  \ln \rho _L \quad \hbox{and}
\quad
\vert w'_R(x) \vert \leq C, \ x \leq 1-\frac{1}{\rho _R }  \ln \rho _R.
\]
In order to establish the main parameter-uniform error bound, we define the following (slightly wider) {\it analytical}  layer widths to be
\begin{equation}\label{tau}
\tau  _L :=  \frac{2}{\rho _L }  \ln \rho _L    \quad \hbox{and} \quad \tau  _R :=  \frac{2}{\rho _R }  \ln \rho _R
\end{equation}
and we choose to measure the accuracy of our numerical approximations in the following weighted $C^1$ norm
\begin{eqnarray}\label{norm}
 \Vert v \Vert _{1,\chi} := \Vert \chi  v'  \Vert   + \Vert v \Vert ,  \quad   \hbox{where} \quad
\chi(x) := \left \{  \begin{array}{lll} \sqrt{\ve \theta} , \  \hbox{if }\quad    x   \leq \tau _L,\\
1,\quad   \quad \hbox{if }\quad  \tau _L <  x  < 1-  \tau _R,  \\
\sqrt{\frac{\ve}{\theta}} , \  \hbox{if}\quad   x \geq 1-\tau _R,
\end{array} \right .
\end{eqnarray}

\section{Discrete Problem}

On the domain  $\Omega $ a piecewise-uniform Shishkin mesh \cite{fhmos}  of $N$ mesh intervals
is constructed as follows. The domain $[0,1]$ is subdivided into three subintervals:
\begin{subequations}\label{scheme1}
\begin{equation}
[ 0,  \sigma _L] \cup  \ [ \sigma _L,1-\sigma _R ] \cup [1-\sigma _R,1],
\end{equation}
where the transition parameters between the subintervals are taken to be
\begin{equation}\label{sigma1}
\sigma  _L := \min \left \{ \frac{1}{4}, \ \frac{2 }{\rho _L}  \ln N \right \} , \quad \sigma  _R := \min \left \{ \frac{1}{4}, \ \frac{4 }{\rho _R} \ln N \right \}.
\end{equation}
\end{subequations}
 Throughout most of the analysis in this paper we shall deal with the case where
\begin{equation}\label{assumption}
\sigma _R \leq \sigma _L <1/4.
\end{equation}
 On each of the two end  subintervals a uniform mesh with $\frac{N}{4}$ mesh-intervals is placed. The remainder of the mesh points are placed in the inner coarse mesh region. Throughout the paper, the mesh step $h_i:=x_i- x_{i-1}$ and $h_L,H,h_R$ denote the  mesh width in the left fine mesh, the central coarse mesh and the right fine mesh, respectively.

The subsequent layer-adapted  piecewise uniform mesh will be denoted by $\omega ^N_{\ve, \mu}$. By this choice of transition parameters, we see that
\[
h ^k_L \vert w_L \vert _k \leq C(N^{-1}\ln N)^k, \quad
h ^k_R \vert w_R \vert _k \leq C(N^{-1}\ln N)^k, \quad k=1,2.\]
The discrete problem is of the form:
\begin{subequations}
\begin{eqnarray}\label{upwindop}
L^{N}U(x_{i}) &=& \bigl( -\ve \delta^{2}  + \mu a D^{-}  + b\bigr)
U (x_i)=f(x_i), \qquad x_{i} \in \omega ^N_{\ve, \mu} ; \\
&&U (0)= u(0), \quad U (1)= u(1);
\end{eqnarray}
\end{subequations}
where $D^-$ denotes the backward difference operator and $\delta ^2$ is the standard replacement to the second derivative on a non-uniform mesh. \footnote{The
finite difference operators  $D^+,D^-,  \delta^2$ are, respectively,  defined by
\[
D^+Z(x_i) :=\frac{Z (x_{i+1})-Z(x_i)}{h_{i+1}};\quad  D^-Z(x_i) :=D^+Z(x_{i-1}) ; \
 \delta^2 Z(x_i) :=\frac{D^+Z(x_{i})-D_r^-Z(x_i)}{(x_{i+1}-x_{i-1})/2}.\]}

 Analogous to the continuous solution, the discrete solution can be decomposed into the sum  $U = V+W_{L}+W_{R}$,
where the components are the solutions of the problems
\begin{subequations}
\begin{eqnarray}
(L^{N}V) (x_i) &=& f(x_i), \qquad V(0) = v(0),\ V(1)=v(1);
\label{eqndv}\\
(L^{N}W_{L})(x_i) &=& 0, \qquad W_{L}(0) =
w_{L}(0), \quad W_{L}(1) =
w_{L}(1);
\label{eqndwl}\\
 (L^{N}W_{R})(x_i) &=& 0, \qquad  W_{R}(0) =
w_{R}(0)=0, \quad W_{R}(1) =
w_{R}(1). \label{eqndwr}
\end{eqnarray}
\end{subequations}
In the next result, we establish bounds on the discrete layer components, which are the discrete counterparts to the bounds (\ref{exp-bound-w}a) established on the continuous layer components.

\begin{theorem}\label{t1} Assume (\ref{assumption}).
We have the following bounds on $W_{L}$ and $W_{R}$
\begin{subequations}
\begin{eqnarray}
\mathopen|{W_{L}}(x_{j})\mathclose| &\leq&
C\prod_{i=1}^{j}(1+\rho _{L}h_{i})^{-1} =: \Psi_L (x_j), \qquad
\Psi_L(0) = C \label{bounddwl}
\\
\mathopen|{W_{R}}(x_{j})\mathclose| &\leq&
C\prod_{i=j+1}^{N}(1+0.5\rho _{R}h_{i})^{-1} =: \Psi_R(x_j), \qquad
\Psi_R(1) = C. \label{bounddwr}
\end{eqnarray}
\end{subequations}
\end{theorem}
\begin{proof} (i) We begin  with the left boundary layer function $W_{L}$. Recall that $w_L(1) \neq 0$ when $\theta \neq 1$. In this special case, observe that
\[
\rho _L = \frac{\gamma}{2 \mu}, \quad \hbox{if} \ \ \theta \neq 1.
\]
From this and the inequality $e^{-x} \leq (1+x)^{-1}, x >0$, one can deduce that $\Psi_L(1) \geq C e^{-\frac{\gamma}{2 \mu}}$  when $\theta \neq 1$. Hence for all $\theta$, $ \Psi_L(1) \geq \vert w_L(1) \vert$ and $ \Psi_L(0) \geq \vert w_L(0) \vert$. Next we consider the interior mesh points.

Consider $\Phi_{L}^{\pm}(x_{j}) :=
\Psi_L (x_j) \pm W_{L}(x_{j})$, where $\Psi_L (x_j)$ is defined in (\ref{bounddwl}). We have
$L^{N}\Phi_{L}^{\pm}(x_{j}) =-\ve \delta^2\Psi_L (x_j) +
\mu a D^{-} \Psi_L (x_j) + b \Psi_L (x_j)$. Using the properties
\begin{eqnarray*}
\Psi_L (x_j) >0, D^{-} \Psi_L (x_j) = - \rho _{L}
\Psi_L (x_j) <0,\quad \hbox{and} \\ \delta^{2} \Psi_L (x_j) =
\rho _{L}^2 \Psi_L (x_{j+1})\frac{h_{j+1}}{\bar{h_{j}}}>0, 
\end{eqnarray*}
we obtain
\begin{equation*}
L^{N}\Phi_{L}^{\pm}(x_{j}) = -\ve {\rho _{L}}^2
\Psi_L (x_{j+1})\frac{h_{j+1}}{\bar{h_{j}}}- \mu a \rho _{L}
\Psi_L(x_j) + b \Psi_L(x_j).
\end{equation*}
 Rewriting we have
\begin{equation*}
L^{N}\Phi _L^{\pm}(x_j) \geq \Psi_L (x_{j+1})\left({2\ve
{\rho _{L}}^2 \left(1-\frac{h_{j+1}}{2\bar{h_{j}}}\right)}+
{(b - 2\ve {\rho _{L}}^2 - \mu a \rho _{L} } + (b-  \mu a \rho _{L})
\rho _{L}h_{j+1} \right).
\end{equation*}

Note that
\begin{eqnarray*}
(b-  \mu a \rho _{L}) = a(\frac{b}{a} -  \mu  \rho _{L})&\geq&  a \gamma 0.5 \quad \hbox{and} \\
b - 2\ve {\rho _{L}}^2 - \mu a \rho _{L} = a\bigl(\frac{b}{a} - \frac{\gamma \alpha}{2\theta a} - \frac{\mu}{2} \sqrt{\frac{\gamma \alpha}{\theta \ve }}\bigr) &\geq&
a\gamma \bigl(1- \frac{1}{2\theta } -  \frac{1}{2 }\sqrt{\frac{ \alpha \mu ^2}{\gamma\theta \ve }}\bigr) \geq 0.
\end{eqnarray*}
 Now using the discrete minimum principle we obtain the required bound (\ref{bounddwl}).

(ii) The same argument is applied to bound $W_{R}$. Consider
$\Phi_{R}^{\pm}({x_j}) = \Psi_R(x_j) \pm
W_{R}(x_{j})$, where $\Psi_R(x_j)$ is defined in (\ref{bounddwr}). Then we have
\[
L^{N}\Phi_{R}^{\pm}({x_j}) =-\ve \delta^2 \Psi_R(x_j) +
\mu a D^{-} \Psi_R(x_j) + b \Psi_R(x_j),
\] and using
\begin{equation*}
 D^{+}
\Psi_R(x_j) = 0.5\rho_{R} \Psi_R(x_j), \textrm{ and } \delta^{2}
\Psi_R(x_j) = \frac{{\rho_{R}}^2}{4(1+0.5\rho_{R}h_{j})}
\Psi_R(x_j)\frac{h_{j}}{\bar{h_{j}}},
\end{equation*}
we obtain
\begin{equation*}
 L^{N}\Phi_R ^{\pm}(x_{j}) \geq
\Psi_R(x_{j-1})\bigg(-0.5\ve {\rho_{R}}^2
+ (b + 0.5\mu a
\rho _{R} )(1+0.5\rho_{R}h_{j})\bigg) \geq 0.
\end{equation*}
We complete the argument using the discrete minimum principle to obtain the required bound
(\ref{bounddwr}).
\end{proof}
From these bounds we deduce that, for all $x_i \leq 1-\sigma _R$,
\begin{subequations}
\begin{eqnarray}
\vert W_R(x_i) \vert \leq  \vert W_R(1-\sigma _R) \vert &\leq& (1+\sqrt{\frac{\theta \gamma \alpha}{4 \ve }}h_{R})^{-N/4} \nonumber \\
&\leq& (1+\frac{8\ln N}{N })^{-N/4} \leq CN^{-2}; \label{wR-outside}
\end{eqnarray}
and, at the left end, for all $x_i \geq \sigma _L$
\begin{eqnarray}
\vert W_L(x_i) \vert \leq  \vert W_L(\sigma _L) \vert &\leq& (1+\sqrt{\frac{\gamma \alpha}{4\theta \ve }}h_{L})^{-N/4}\nonumber \\
&\leq& (1+\frac{8\ln N}{N })^{-N/4} \leq CN^{-2}. \label{wL-outside}
\end{eqnarray}
\end{subequations}
Hence, outside their corresponding layer regions, the discrete layer functions $W_L,W_R$ are small, from a computational perspective.
\section{Nodal error analysis}

We denote the nodal error and associated truncation error, respectively,  by  \[
e(x_i):=U(x_i)-u(x_i), \quad \hbox{and} \quad
 {\cal T}(x_i):=L^{N} e(x_i).\] When bounding the local truncation error, we  utilize the following standard bounds at all mesh points, excluding the transition points: For all  $x_i \neq \sigma _L, 1- \sigma _R$
\begin{eqnarray}\label{TE1}
&&\vert L^N(U-u)(x_i)) \vert \leq \nonumber \\ &&C h_i \Bigl( \ve \max \{ \Vert u^{(3)} \Vert _{[x_{i-1},x_{i+1}]}, h_i \Vert u^{(4)} \Vert _{[x_{i-1},x_{i+1}]} \}+  \mu  \Vert u^{(2)} \Vert _{[x_{i-1},x_{i}]}\Bigr), \
\end{eqnarray}
and at all mesh points
\begin{eqnarray}\label{TE2}
\vert L^N(U-u)(x_i)) \vert \leq C \ve (h_i +h_{i+1})\Vert u^{(3)} \Vert _{[x_{i-1},x_{i+1}]} + C \mu h_i \Vert u^{(2)} \Vert _{[x_{i-1},x_{i}]}.
\end{eqnarray}
We define the discrete error flux to be
$$
{\cal U}^-_{i}:=D^- e(x_i), \quad \hbox{ if } 0< x_i\le 1 \quad \hbox{and} \quad {\cal U}^+_{i}:=D^+ e(x_i), \quad \hbox{ if } 0\leq x_i< 1.
$$
On a piecewise-uniform mesh the finite difference operators $\delta^2$ and $D^-$ do not commute on a non-uniform mesh. Based on this observation, we define a new finite difference operator $\hat \delta ^2$ by
\begin{equation}\label{hat-delta}
\hat \delta ^2 Z_{i}:=
\ds\frac{1}{\hbar_i}\bigl(\ds\frac{h_{i+1}}{h_{i}}D^+
-\ds\frac{\hbar_{i}}{\hbar_{i-1}} D^-  \bigr) Z_{i},
\end{equation}
which has the property that 
\[
\hat \delta ^2 D^- Z_{i} \equiv D^- \delta ^2  Z_{i}
\]
on an arbitrary mesh. Note that the second order operator is $\hat \delta ^2$ on the left and $ \delta ^2$ on the right of this identity. Hence, this identity is not a statement of commutativity. 
Note the following identity (Discrete derivatives of a product of two mesh functions)
\begin{equation}\label{discrete-product-rule}
D^-(P_{i}Q_{i}) \equiv P_{i} D^- Q_{i} +Q_{i-1} D^-P_{i}.
\end{equation}
Using these identities and $D^-(L^{N} e(x_i)) = D^-{\cal T}(x_i)$, we see that for all mesh points within the region $(h_1,1)$, the discrete flux ${\cal U}^-_{i}$ satisfies
\begin{eqnarray}\label{prob-deriv}
\hat L^{N} {\cal U}^-(x_{i}) &=&  D^- {\cal T}(x_i)-e(x_{i-1}) D^- b(x_i), \quad  x_i \in (h_1,1) 
\end{eqnarray}
where for the internal mesh points
\begin{equation} \label{mod-discrete-operator}
 \hat L^{N}Z(x_{i}):=(-\ve \hat \delta ^2
+\mu a(x_{i-1}) D^-
+(b+\mu D^- a)(x_i)I)Z(x_{i}),  \end{equation}
and for the end points $\hat L^{N}Z(x_{i}):=Z(x_{i})$ for  $x_i=h_1,1$.

 Note the following classical bounds on the truncation error:
\begin{eqnarray*}
 (D^-y - y ')(x_i) &=& \frac{1}{h_i}
\int _{t=x_{i-1}}^{x_i} y'(t) - y'(x_i) \ d \ t \\ &=&
\frac{1}{h_i}  \int _{t=x_{i-1}}^{x_i} \int
_{s=x_{i}}^{t}y''(s) \ d  \ s \  \ d \ t .\\  \\
 D^-(D^-y - y ')(x_i)
&=&   \frac{1}{h^2_i}  \int _{t=x_{i-1}}^{x_i} \int
_{s=x_{i}}^{t}y''(s) -y''(x_{i-1}) \ d   s \   d  t \\ &-& \frac{1}{h_ih_{i-1}}  \int _{t=x_{i-2}}^{x_{i-1}} \int
_{s=x_{i-1}}^{t}y''(s) -y''(x_{i-1})\ d   s \   d  t \\
&+& y''(x_{i-1}) \frac{1}{2}(1-\frac{h_i}{h_{i-1}}). \\
 (\delta ^2 y - y '')(x_i) &=& \frac{1}{\bar h_{i}}
 \bigl(\frac{1}{h_{i+1}}  \int _{t=x_{i}}^{x_{i+1}} \int
_{s=x_{i}}^{t}y''(s) -y''(x_i) \ d  \ s \  \ d \ t - \\ &&
\frac{1}{h_i} \int _{t=x_{i-1}}^{x_i} \int
_{s=x_{i}}^{t}y''(s) -y''(x_i) \ d  \ s \  \ d \ t \bigr).
\end{eqnarray*}
Based on these bounds, we have that at any mesh point,
\begin{subequations}\label{trunc_x}
\begin{eqnarray}
\vert D^-(u' - D^-u) (x_i) \vert &\leq& C(1+\frac{h_{i-1}}{h_i}) \Vert u^{(2)} (x) \Vert _{x \in (x_{i-2},x_i)}, \\
\vert D^-(u'' - \delta ^2 u) (x_i) \vert &\leq& C (1+\frac{h_{i-1}+h_{i+1}}{h_i})\Vert u^{(3)} (x) \Vert _{x \in (x_{i-2},x_{i+1})}.
\end{eqnarray}
In addition, if $ h_{i-1}=h_i$, then
\begin{equation}
\vert D^-(u' - D^-u) (x_i) \vert \leq Ch_i \Vert u^{(3)} (x) \Vert _{x \in (x_{i-2},x_i)},
\end{equation}
and if $h_{i-1}=h_i=h_{i+1}$, then
\begin{equation}
\vert D^-(u'' - \delta ^2 u) (x_i) \vert \leq C\max \{ h_i \Vert u^{(4)} (x) \Vert _{x \in (x_{i-1},x_{i+1})}, h^2_i \Vert u^{(5)} (x) \Vert _{x \in (x_{i-2},x_{i+1})} \}.
\end{equation}
\end{subequations}
Based on the assumption (\ref{assum}) the discrete operator $\hat L^{N}$ (\ref{mod-discrete-operator}) satisfies a discrete comparison principle. To bound the  error in the discrete flux ${\cal U}^-_{i}$, we employ a standard stability and consistency argument using the operator $\hat L^{N}$ (and not the operator $L^N$). To this end we bound $D^-(L^N(e(x_i)))$ and the error fluxes at the endpoints of the interval $(h_1,1)$. The main complication in the analysis is the construction of  suitable discrete barrier functions. 

Now we deduce bounds on the regular ${\cal V}^-:=D^-(V-v)$ and the singular  components  ${\cal W}_L^-:=D^-(W_L-w_L), {\cal W}_R^-:=D^-(W_R-w_R)$ of the discrete error flux ${\cal U}^-$. We begin with the singular component $W_L$ as in this case the analysis is a little easier. We will need an appropriate  bound on the boundary error flux $\vert D^+(W_L-w_L)(0)\vert $. We achieve this by sharping the standard nodal error bound $\vert (W_L-w_L)(x_i) \vert \leq CN^{-1} \ln N$, within the layer region on the left, to reflect the fact that $(W_L-w_l)(0)=0$.

\begin{lemma} \label{l_bf_sc2} Assume (\ref{assumption}).
 For  sufficiently large $N$,
\begin{equation}\label{outflow-bnd2}
\sqrt{\ve \theta}\vert D^+(W_L-w_L)(0) \vert \le C N^{-1} (\ln N),
\end{equation}
where $W_L$ is the  solution of (\ref{eqndwl}) and $w_L$ is the  solution of (\ref{eqnwl1}).
\end{lemma}
\begin{proof}
The proof splits into the two cases of $\theta > 1$ and $\theta=1$.

(i) In the convection-diffusion case of $\theta>1$, we introduce  the following linear discrete barrier function
$$
B(x_i):=C\frac{ x_i}{\mu}\Vert L^{N}(W_L-w_L)\Vert_{(0,\sigma_L)} +CN^{-2},
$$
so that $L^NB \geq C\Vert L^{N}(W_L-w_L)\Vert$. Note that this barrier function cannot be used in the reaction-diffusion case when $\theta =1$, as it involves the multiple $\mu ^{-1}$. Here $\Vert L^{N}(W_L-w_L)\Vert _{(0,\sigma_L)}$ is the truncation error associated with the left singular component $w_L$. In the boundary layer region $(0,\sigma_L)$, using (\ref{exp-bound-w}c) and the standard truncation error bounds (\ref{TE1}) we have that
$$
\Vert L^{N}(W_L-w_L)\Vert_{(0, \sigma_L)}\le C (\frac{1}{\theta}+\frac{\mu}{\sqrt{\ve\theta}}) N^{-1}\ln N.
$$
In addition,  by (\ref{exp-bound-w}a) and (\ref{wL-outside}) we can deduce that $(W_L-w_L)(0)=0$ and  $\vert (W_L-w_L)(\sigma _L) \vert \leq CN^{-2}$. From the discrete minimum principle, we then have that, for $x_i\in [0,\sigma_L]$,
\[
\vert (W_L-w_L)(x_i) \vert\le B(x_i) \leq C x_i(\frac{1}{\mu\theta}+ \frac{1}{\sqrt{\ve\theta}}) N^{-1}\ln N  +CN^{-2}
\]
and, in particular,
\[
\vert (W_L-w_L)(h_L) \vert\leq C h_L(\frac{1}{\mu\theta}+ \frac{1}{\sqrt{\ve\theta}}) N^{-1}\ln N +CN^{-2}.
\]
Therefore, when $\theta > 1$,
\begin{eqnarray*}
\sqrt{\ve\theta}\vert D^+(W_L-w_L)(0)\vert &=& \frac{\sqrt{\ve\theta}  \vert (W_L-w_L)(h_L) \vert}{h_{L}} \\
& \le&   C\sqrt{\ve\theta}(\frac{1}{\theta \mu}+ \frac{1}{\sqrt{\ve\theta}}) N^{-1}\ln N +CN^{-1} \\
&\le& C N^{-1}(\ln N).
\end{eqnarray*}

(ii) In the reaction-diffusion case, where  $\theta = 1$, we utilize the bound (\ref{wL-outside}) to allow us confine   the  truncation error estimate (\ref{TE1}) to the fine uniform mesh. For all mesh points $x_i \in (0,\sigma _L)$, this yields
\begin{eqnarray*}
&&\vert L^{N}(W_L-w_L) (x_i)\vert \leq \ve(h_L)^2||w_{L}^{(4)}||+\mu(h_{L})||w_{L}^{(2)}||\\
&&\leq C\left(\ve(\sqrt{\ve})^2(N^{-1}\ln N)^2\left(\frac{1}{\sqrt{\ve}}\right)^4+\mu(\sqrt{\ve})(N^{-1}\ln N)\left(\frac{1}{\sqrt{\ve}}\right)^2\right)\\
&&\leq C\left(N^{-1}\ln N+\frac{\mu}{\sqrt{\ve}}\right)(N^{-1}\ln N).
\end{eqnarray*}
Consider the following discrete barrier function
\[
C\left(\frac{N^{-1}\ln N}{\beta}+\frac{x_i}{\sqrt{\ve}}\right) (N^{-1}\ln N)+CN^{-2}
\]
and using the discrete minimum principle we get that
\[
\vert (W_L-w_L)(h_L) \vert \leq C\left(\frac{(N^{-1}\ln N)^2}{ \beta}+(N^{-1}\ln N)\frac{1}{\sqrt{\ve}}h_L\right) +CN^{-2}.
\]
Now we have, for the case when $\theta =1$,
\begin{eqnarray*}
\sqrt{\ve}\vert D^+(W_L-w_L)(0)\vert &=& \sqrt{\ve} \frac{\vert (W_L-w_L)(h_L)\vert }{h_L} \\
& \le&   C\sqrt{\ve}\left(\frac{(N^{-1}\ln N)^2}{\beta}+(N^{-1}\ln N)\frac{1}{\sqrt{\ve}}h_L\right)\frac{1}{h_L} +CN^{-1}\\
& \le&  CN^{-1}\ln N.
\end{eqnarray*}
Hence we have completed the proof for both $\theta >1$ and $\theta=1$.
\end{proof}
Note that by examining the bounds in the above Lemma, we have the nodal error bound
\begin{equation}\label{NodalError-W_L}
\vert (W_L-w_L)(x_i) \vert \leq C N^{-1} (\ln N)^2.
\end{equation}
\begin{theorem}\label{LWBound} Assume (\ref{assumption}).  We have the bounds
\begin{eqnarray*}
\sqrt{\ve\theta}|D^-(W_L-w_L)(x_i)| &\leq& CN^{-1}\ln N, \quad \hbox{ if} \quad  0 < x_i  \leq \sigma_L\\
|D^-(W_L-w_L)(x_i)| &\leq& CN^{-1}, \quad \hbox{ if} \quad  \sigma_L< x_i \leq 1-\sigma_R\\
\sqrt{\frac{\ve}{\theta}}|D^-(W_L-w_L)(x_i)| &\leq& CN^{-1}, \quad \hbox{ if} \quad   1-\sigma_R < x_i \leq 1;
\end{eqnarray*}
where $W_L$ is the  solution of (\ref{eqndwl}) and $w_L$ is the   solution of (\ref{eqnwl1}).
\end{theorem}
\begin{proof}
Using the bounds (\ref{exp-bound-w}a) and (\ref{wL-outside}), respectively, on $w_L$ and $W_L$ we see that outside the left layer region
\[
|(W_{L}-w_{L})(x_i)|\leq CN^{-2}, \quad x_i \geq \sigma _L.
\]
Combining this bound  with the fact that $h_R = C\sqrt{\frac{\ve}{\theta}}N^{-1}\ln N$ we deduce that
\begin{eqnarray*}
|D^-(W_L-w_L)(x_i)| &\leq& CN^{-1} \quad \hbox{if} \quad  \sigma_L < x_i \leq 1-\sigma_R\\
\sqrt{\frac{\ve}{\theta}}|D^-(W_L-w_L)(x_i)| &\leq& CN^{-1} \quad \hbox{if} \quad x_i  > 1-\sigma_R.
\end{eqnarray*}
It remains to establish the bound in the left layer region, where the derivatives of the left boundary layer function $w_L$ are significant.
From (\ref{exp-bound-w}a) we have that
\begin{eqnarray*}
|w_L(\sigma_L-h_L)| \leq Ce^{\frac{\sqrt{\gamma \alpha}h_L}{ 2\sqrt{\ve \theta}}}e^{\frac{-\sqrt{\gamma \alpha}\sigma_L}{2 \sqrt{\ve \theta}}}\leq CN^{-2};
\end{eqnarray*}
and using Theorem 3, with $\rho _{L} := \sqrt{\frac{\gamma \alpha}{4\theta \ve }}$ it follows that
\begin{equation}
\mathopen|{W_{L}}(\sigma_L-h_L)\mathclose| \leq
C(1+\rho _{L}h_{L})(1+\rho _{L}h_{L})^{-\frac{N}{4}} \leq CN^{-2}.
\end{equation}
Repeat the earlier argument to get that
\[\sqrt{\ve \theta}\vert D^-(W_L-w_L) (\sigma _L)\vert\leq C\frac{\sqrt{\theta \ve}}{h_L}N^{-2} \leq CN^{-1}. \]
Using the truncation error bounds $\eqref{trunc_x}$ in the region $(0, \sigma_L)$ we have
\begin{eqnarray*}
 \vert \hat{L}^{N}D^-(W_L-w_L) (x_i)\vert&\leq& Ch_L(\ve h_L||w_L^{(5)}||+\mu||a_x||||w_L^{(2)}||+ \mu||a||||w^{(3)}_L||)\\ \quad &&+ \qquad CN^{-1}\ln N \\
&\leq& \frac{C}{\sqrt{\ve\theta}}N^{-1}\ln N.
\end{eqnarray*}
Complete the proof  using the discrete constant barrier function $\frac{N^{-1}\ln N }{\sqrt{\ve\theta}}$, Lemma 4, the lower bound $b> \gamma \alpha$ and the end-point bound
of
\[ \sqrt{\ve\theta} \vert  D^-(W_L-w_L) (\sigma _L) \vert \leq CN^{-1}.\]
\end{proof}

The analysis is more elaborate in the case of the right layer component $w_R$. We first need an appropriate  bound on the outgoing error flux $\vert D^-(W_R-w_R)(1)\vert $. We again achieve this by sharping the standard nodal error bound $\vert (W_R-w_R)(x_i) \vert \leq CN^{-1} \ln N$, within the layer region on the right, to reflect the fact that 
$(W_R-w_R)(1)=0$.

\begin{lemma} \label{l_bf_sc} Assume (\ref{assumption}).
 For  sufficiently large $N$,
\begin{equation}\label{outflow-bnd}
\sqrt{\frac{\ve}{\theta} }\vert D^-(W_R-w_R)(1) \vert \le C N^{-1} (\ln N)^2,
\end{equation}
where $W_R$ is the  solution of (\ref{eqndwr}) and $w_R$ is the solution of (\ref{eqnwr1}).
\end{lemma}
\begin{proof}  Consider the discrete function $\psi(x_i)$ defined  by
\begin{eqnarray*}
-\ve \delta^2 \psi +  \sqrt{\ve \theta} A D^- \psi=0, \ x_i\in (1-\sigma _R,1), \\
\psi(1-\sigma _R)=1, \ \psi(1)=0; \qquad A \geq \Vert a \Vert \sqrt{\frac{\gamma}{\alpha}}.
\end{eqnarray*}
Observe that
\[
\psi (x_i) = \frac{1 - (1+\rho)^{i-N} }{1 - (1+\rho)^{-N/4}}, \qquad \hbox{where} \ \rho := \sqrt{\frac{ \theta}{\ve}} A h _R.
\]
Note also that
\[
D^- \psi (x_i) < 0 \qquad \hbox{and} \quad (1+\rho)^{-N/4} \leq (1+\frac{4\ln N}{N})^{-N/4} \leq CN^{-1}.
\]
Hence $\psi (x_i) \leq C(1 - (1+\rho)^{i-N} )$ for $N$ sufficiently large. Now we define a barrier function to deduce appropriate bounds for ${\cal W}^-_{N}$. First, we note that
\[
L^{N}(x_i-1+\sigma _R \psi(x_i))\geq \mu a(x_i) + \sigma _R (\mu a(x_i) - \sqrt{\ve \theta} A) D^- \psi(x_i) \geq \mu \alpha .
\]

(i) When $\theta >1$, define  the following discrete barrier function
\begin{equation}\label{defB}
B(x_i):=C\mu ^{-1}\Vert L^{N}(W_R-w_R)\Vert_{(1-\sigma _R)} \bigl(x_i-1+\sigma _R\psi(x_i) \bigr) +CN^{-2},
\end{equation}
where $L^{N}(W_R-w_R)$ is the truncation error associated with the singular component $w_R$. In the boundary layer region
$$
\Vert L^{N}(W_R-w_R)\Vert_{(1-\sigma _R, 1)}\le C (\theta+\mu \sqrt{\frac{\theta}{\ve}}) N^{-1}\ln N.
$$
Using the discrete maximum principle we then have that, for $x_i\in [1-\sigma _R,1]$
\[
\vert (W_R-w_R)(x_i) \vert\le B(x_i) \leq C (\frac{\theta}{\mu} + \sqrt{\frac{\theta}{\ve}}) N^{-1}\ln N\bigl(x_i-1+\sigma _R\psi(x_i) \bigr) +CN^{-2}
\]
and
\[
\vert (W_R-w_R)(1-h_R) \vert\leq C h_R(\frac{\theta}{\mu} + \sqrt{\frac{\theta}{\ve}}) N^{-1}(\ln N)^2 +CN^{-2}.
\]
Therefore, when $\theta > 1$,
\begin{eqnarray*}
\sqrt{\frac{\ve}{\theta} } \vert {\cal W}^-_{N}\vert &=& \frac{\sqrt{\ve}}{h_R\sqrt{\theta} } \vert (W_R-w_R)(1-h_R) \vert \\
& \le&  C\frac{\sqrt{\ve}}{\sqrt{\theta} } (\frac{\theta}{\mu} + \sqrt{\frac{\theta}{\ve}}) N^{-1}(\ln N)^2 +CN^{-1} \\
&\le& C N^{-1}(\ln N)^2.
\end{eqnarray*}

(ii) In the other case, where $\theta =1$, we can use the  truncation error bound (\ref{TE1}) in the boundary layer region $(1-\sigma _R,1)$,
\begin{eqnarray*}
\vert L^N(W_R-w_R)(x_i) \vert &\leq& C\left(\ve(h_R)^2||w_{R}^{(4)}||+\mu(h_{(R)})||w_{R}^{(2)}||\right), \quad x_i \in (1-\sigma _R,1)\\
&\le& C ( N^{-1}\ln N+ \frac{\mu }{\sqrt{\ve}}) N^{-1}\ln N.
\end{eqnarray*}
Using the barrier function
\begin{equation}\label{defBnew}
C\left(\frac{(N^{-1}\ln N)^2}{\beta}+{\frac{N^{-1}\ln N}{\sqrt{\ve}}} \bigl(x_i-1+\sigma _R\psi(x_i) \bigr)\right),
\end{equation}
we get
\[
\vert (W_R-w_R)(x_i) \vert\le C\frac{ ( N^{-1}\ln N)^2}{\beta} +\bigl(x_i-1+\sigma _R\psi(x_i) \bigr) \frac{N^{-1}\ln N}{\sqrt{\ve}},
\]
which yields the required result for the case of $\theta =1$.
\end{proof}
In passing, we note  that the nodal error bound
\begin{equation}\label{NodalError-W_R}
\vert (W_R-w_R)(x_i) \vert \leq C N^{-1} (\ln N)^2,
\end{equation}
follows from the bounds established in the above Lemma
\begin{theorem}\label{RWBound} Assume (\ref{assumption}). We have the bounds
\begin{eqnarray*}
\sqrt{\ve\theta}|D^-(W_R-w_R)(x_i)| &\leq& CN^{-1}, \quad \hbox{if} \quad  x_i  \leq \sigma_L\\
|D^-(W_R-w_R)(x_i)| &\leq& CN^{-1}, \quad \hbox{if} \quad  \sigma_L< x_i \leq 1-\sigma_R\\
\sqrt{\frac{\ve}{\theta}}|D^-(W_R-w_R)(x_i)| &\leq& CN^{-1}(\ln N)^2, \quad \hbox{if} \quad  x_i  > 1-\sigma_R;
\end{eqnarray*}
where $W_R$ is the   solution of (\ref{eqndwr}) and $w_R$ is the   solution of (\ref{eqnwr1}).
\end{theorem}
\begin{proof}
Using the bounds (\ref{exp-bound-w}a) and (\ref{wR-outside}) on $w_R$ and $W_R$, we see that outside the layer region $(1-\sigma_R,1)$ we have
\[
\vert (W_{R}-w_{R})(x_i) |\leq CN^{-2},\quad  0 \leq x_i \leq 1-\sigma_R.
\]
Using this bound along with the mesh step $h_L = C\sqrt{{\ve}{\theta}}N^{-1}\ln N$, we deduce that
\begin{eqnarray*}
\sqrt{\ve \theta}|D^{-}(W_R-w_R)(x_i)| &\leq& CN^{-1}, \quad \hbox{if} \quad  x_i  \leq \sigma_L\\
|D^-(W_R-w_R)(x_i)| &\leq& CN^{-1}, \quad \hbox{if} \quad   \sigma_L < x_i \leq 1-\sigma_R.
\end{eqnarray*}
When $x_i = 1-\sigma_R+h_R, 1-\sigma_R+2h_R$ we also have
\begin{eqnarray*}
|w_R(x_i)| \leq Ce^{\frac{2\sqrt{\gamma \alpha \theta}h_R}{ \sqrt{\ve }}}e^{-\frac{\sqrt{\gamma \alpha \theta}\sigma_R}{\sqrt{\ve}}}\leq CN^{-4}
\end{eqnarray*}
and using Theorem 3, with $\rho _{R} := \sqrt{\frac{\theta \gamma \alpha}{\ve }}$ we have
\begin{equation}
\mathopen|{W_{R}}(x_{i})\mathclose| \leq
C(1+0.5\rho_{R}h_{R})^2(1+0.5\rho_{R}h_{R})^{-\frac{N}{4}} \leq CN^{-2}.
\end{equation}
We therefore have established that
\[\sqrt{\frac{\ve}{ \theta}}\vert {\cal W_R}^-\vert \leq CN^{-1}, \qquad x_i  = 1-\sigma_R+h_R, 1-\sigma_R+2h_R. \]

In the region $(1-\sigma_R+h_R, 1)$, using the truncation error bounds \eqref{trunc_x} we have
\begin{eqnarray*}
\hat{L}^{N}{\cal W_R^-} &\leq& Ch_R\left(\ve h_R||w_R^{(5)}||_{(x_{i-2},x_{i+1})}+\mu||a_x||||w_R''||_{(x_{i-1},x_i)}+ \mu||a||||w^{(3)}_R||_{(x_{i-2},x_i)}\right) \\
&+& CN^{-1}\ln N.
\end{eqnarray*}
Using the exponential bounds in Theorem  \ref{boundwrightleft} we see that
\begin{eqnarray*}
\hat{L}^{N,M}{\cal W_R^-} &\leq& C\sqrt{\frac{\ve}{\theta}}N^{-1}\ln N\left(\ve\sqrt{\frac{\ve}{\theta}}\left(\sqrt{\frac{\theta}{\ve}}\right)^5+\mu\left(\sqrt{\frac{\theta}{\ve}}\right)^3\right)e^{-\sqrt{\frac{\gamma\alpha\theta}{\ve}}(1-x)} \\ && \quad + \quad CN^{-1}\ln N\\
&\leq& C\sqrt{\frac{\theta}{\ve}}N^{-1}\ln N\left(\theta+\sqrt{\frac{\mu^2\theta}{\ve}}\right)e^{-\sqrt{\frac{\gamma\alpha\theta}{\ve}}(1-x_{i+1})}+CN^{-1}\ln N.
\end{eqnarray*}
In the case of $\theta = 1$, this truncation error bound simplifies to 
 \begin{eqnarray*}
\vert \hat{L}^{N}{\cal W_R^-} \vert
&\leq& C\frac{N^{-1}\ln N}{\sqrt{\ve}}, \quad \hbox{if} \quad \theta =1,
\end{eqnarray*}
and the result follows using  a constant discrete barrier function.

When $\theta>1$,  the truncation error bound is of the form
\begin{eqnarray*}
\vert \hat{L}^{N}{\cal W_R^-} (x_i) \vert &\leq& C\sqrt{\frac{\theta}{\ve}}N^{-1}\ln N\theta e^{-\sqrt{\frac{\gamma\alpha\theta}{\ve}}(1-x_{i+1})}+CN^{-1}\ln N.
\end{eqnarray*}
 Consider the discrete barrier function
\[
\sqrt{\frac{\theta}{\ve}}N^{-1}\ln N(1+(1+h_R\sqrt{\frac{\gamma\alpha\theta}{\ve}})^{i+1-N})
\]
and use the strict inequality $a(x)>\alpha$ and $(1+t)^{-1} \geq e^{-t}$  to get the required result.


\end{proof}

We next move onto the analysis of the error associated with the regular component.

\begin{lemma} For the discrete regular component $V$ and the continuous regular component $v$ we have the bound
\[
\vert D^+(V-v)(0) \vert \le C N^{-1}, \quad
\sqrt{\frac{\ve}{\theta}}  \vert D^-(V-v)(1) \vert \le C N^{-1}.
\]
\end{lemma}

\begin{proof} The proof is given in  Appendix B. \end{proof}

Within the  proof of Lemma 8, one can see that we have established the nodal error bound $\Vert V-v \Vert \leq CN^{-1}x_i + CN^{-2}$. Using the corresponding earlier bounds on the nodal error on the layer components,  we now have the parameter-uniform nodal error bound
\begin{equation}\label{NodalError}
\Vert U-u \Vert _{\Omega ^N}\leq CN^{-1}(\ln N)^2.
\end{equation}

In the next Theorem, the definition of $\tilde \delta ^2$ comes into play into the numerical analysis for the first time, as the consistency bound is derived over the entire (non-uniform) mesh. The use of the operator  $\tilde \delta ^2$ results in isolated jumps in the truncation error at the four mesh points $\sigma _L, \sigma _L +H, 1- \sigma _R+h_R, 1-\sigma _R$.

\begin{theorem}\label{RegBound}
 Assume (\ref{assumption}). We have
\begin{eqnarray*}
\vert D^- (V-v)(x_i) \vert &\le& C N^{-1}, \quad  \hbox{if }\ x_i \leq  \sigma _L,\\
\vert D^- (V-v)(x_i) \vert &\le& C N^{-1} (\ln N)^2,  \hbox{if } 1-\sigma _L < x_i \leq 1-\sigma _R, \\
\sqrt{\frac{\ve}{\theta}} \vert D^- (V-v)(x_i) \vert &\le& C N^{-1} (\ln N)^2, \hbox{if } x_i > 1-\sigma _R,
\end{eqnarray*}
\end{theorem}
\begin{proof} The proof is given in Appendix C. \end{proof}

Given the bounds in Theorems \ref{LWBound}, \ref{RWBound} and \ref{RegBound}, it only remains to remove the scaling factors in certain parts of the layer regions, in the particular case where  the analytical layer width is thinner than the computational layer width. That is, if $\tau _L \leq \sigma _L$ (or $\tau _R \leq \sigma _R$) then we need to remove the scaling factor $\sqrt{\ve \theta}$ (or $\sqrt{\frac{\ve }{\theta}}$) from the bounds in Theorems \ref{LWBound}, \ref{RWBound} and \ref{RegBound} within the region $\tau _L < x_i < \sigma _L$ (or $1-\sigma _R < x_i < 1- \tau _R$).

 \begin{theorem}\label{NodalBound} Assume (\ref{assumption}).
We have the scaled nodal error  bounds
\begin{eqnarray*}
\sqrt{\ve\theta}\vert D^- (U-u)(x_i) \vert &\le& C N^{-1}, \quad  \hbox{if }\ x_i \leq  \tau _L,\\
\vert D^- (U-u)(x_i) \vert &\le& C N^{-1} (\ln N)^2, \quad \hbox{if } \ 1-\tau _L < x_i \leq 1-\tau _R, \\
\sqrt{\frac{\ve}{\theta}} \vert D^- (U-u)(x_i) \vert &\le& C N^{-1} (\ln N)^2,\quad \hbox{if }\  x_i > 1-\tau _R,
\end{eqnarray*}
where $\tau _l, \tau _R$ are defined in (\ref{tau}).
\end{theorem}
\begin{proof} 

(i) We begin by examining the error in the layer function $W_L$ ($W_R$) in the fine mesh region on the right-hand (left-hand) side of the domain., Let us first consider the error in the left layer function $W_L-w_L$ in the right layer region $(1-\sigma _R, 1-\tau _R)$. For $x \geq 0.5$ and $x_i \geq 0.5$
\begin{eqnarray*}
|w_{L}(x)| &\leq& C e^{-2\rho _L\sigma _L}e^{-\rho _L(x-2\sigma _L)} \leq CN^{-4}; \\
|W_{L}(x_i)| &\leq& C (1+\rho _L h_L) ^{-N/2}\leq CN^{-4}, \quad \hbox{as} \quad H \geq h_L.
\end{eqnarray*}
If $\tau _R \leq \sigma _R$, then $\sqrt{\frac{\theta}{\ve}} \leq CN^2$
and so \[
\vert (w_L - W_L)(x_i) \vert \leq CN ^{-4} \leq CN ^{-2}\sqrt{\frac{\ve}{\theta}},\quad \hbox{if} \quad x_i \geq 0.5. \]
Hence,
\[
\vert D^-(w_L - W_L)(x_i) \vert \leq CN ^{-1},\quad \hbox{if} \quad x_i \geq 0.5 \quad \hbox{and} \quad \tau _R \leq \sigma _R .
\]
An analogous argument can be used to establish
\[
\vert D^-(w_R - W_R)(x_i) \vert \leq CN ^{-2},\quad \hbox{if} \quad x_i \geq 0.5 \quad \hbox{and} \quad \tau _L \leq \sigma _L .
\]
(ii) Let us next consider  the left layer error $D^-(W_L-w_L)$ in the left region $[0, \sigma_L+H]$. A more refined analysis (to that used in Theorem \ref{LWBound})) is required. The analysis requires the construction of a discrete barrier function across the non-uniform mesh and using a sharper truncation error analysis. 
Using the truncation error bounds $\eqref{trunc_x}$ and the exponential bounds in Theorem  \ref{boundwrightleft} in the region $(0, \sigma_L+H)$, we have
\begin{eqnarray*}
 \vert \hat{L}^{N}D^-(W_L-w_L) (x_i)\vert&\leq& \frac{C}{\sqrt{\ve\theta}}N^{-1}\ln N e^{-\rho _L x_i}, \quad x _i < \sigma _L\\  \vert \hat{L}^{N}D^-(W_L-w_L) (x_i)\vert
&\leq& C(\frac{1}{\ve\theta^2}+{\frac{\mu}{\ve\theta}}N^{-1}\ln N) e^{-\rho _L x_i}, \quad x _i = \sigma _L.
\end{eqnarray*}
We now construct a suitable barrier function (which is similar to $\Psi _L$ defined in (\ref{bounddwl})):
\[
\Psi _1 (x_i) := (1+0.5\rho_L h_L)^{-i}, 0 \leq x _i \leq \sigma _L; \qquad \Psi _1 (\sigma_L+H) :=0.
\]
For $x_i < \sigma _L$, as in Theorem 3, $\hat L^N\Psi _1 (x_i) \geq C e^{-0.5\rho _L x_i}$ and for $x_i=\sigma _L$, using (\ref{hat-delta}) and (\ref{assum}),
\begin{eqnarray*}
\hat L^N\Psi _1 (\sigma _L) &=& \Bigl[ \frac{\ve}{h_L} \Bigl(\frac{1}{\bar h_L} - \frac{\rho _L}{2} \Bigr)+ \Bigl(-0.5\rho _L \mu  a(x_{i-1}) + (b +\mu  D^-a)(x_i) \Bigr)\Bigr] \Psi _1 (\sigma _L) \\
&\geq& \Bigl[ \frac{\ve}{h_L\bar h_L} \Bigl(1 - \frac{\bar h_L\rho _L}{2} \Bigr)+ \Bigl(-\frac{\gamma}{4}  a(x_{i-1}) + (b +\mu  D^-a)(x_i) \Bigr)\Bigr] \Psi _1 (\sigma _L) \\
&\geq& \Bigl[ \frac{\ve}{h_L\bar h_L} \Bigl(1 - \frac{\bar h_L\rho _L}{2} \Bigr)+ \Bigl(\frac{b}{2} -\frac{\gamma}{4}  a(x_{i-1}) + 0.5(b +2\mu  D^-a)(x_i) \Bigr)\Bigr] \Psi _1 (\sigma _L) \\
&\geq&  \frac{\ve}{h_L\bar h_L} \Bigl(1 - \frac{\bar h_L\rho _L}{2} \Bigr) \Psi _1 (\sigma _L) \\
&\geq&  \frac{\ve}{2h_L\bar h_L} \Bigl(1 - \frac{\rho _L}{N}(1-\sigma _R) + 1 - \frac{2\ln N}{N}\Bigr) \Psi _1 (\sigma _L) 
\end{eqnarray*}
for $N$ sufficiently large. In the case where $\tau _L \leq \sigma _L$, then  $\rho _L \leq N$ and hence
\[
\hat L^N\Psi _1 (\sigma _L) \geq 0,\quad \hbox{if} \quad \tau _L \leq \sigma _L. 
\]
Consider the  piecewise linear barrier function
\[
\Psi _2 (x_i) := \frac{x_i}{\sigma _L}, 0 \leq x _i \leq \sigma _L; \qquad \Psi _2 (\sigma_L+H) :=1.
\]
For $x_i < \sigma _L$, $\hat L^N\Psi _2 (x_i) \geq 0$ and at the transition point $\sigma _L$, using (\ref{hat-delta}),
\[
\hat L^N\Psi _2 (\sigma _L) = \frac{1}{h_L\sigma _L} (\mu a h _L+ \ve)\geq \frac{CN}{\theta} (\ln N)^{-2}+ C{\frac{\mu}{\sqrt{\ve\theta}}}(\ln N)^{-1}.
\]
Then we deduce that
\[
\vert D^-(W_L-w_L) (x_i)\vert \leq \frac{C}{\sqrt{\ve\theta}}N^{-1}\ln N \Psi _1 (x_i) +C\bigl(\frac{\theta}{N}\frac{1}{\ve\theta^2}
+  \frac{N^{-1}}{\sqrt{\ve \theta}}\bigr){e^{-\rho _L\sigma _L}\Psi _2(x_i)(\ln N)^{2}}.
\]
For $\tau _L \leq x_i \leq \sigma _L$,  noting $e^{-\rho _L\tau _L} \leq C \rho _L ^{-2} \leq C\ve \theta $,
\[
\vert D^-(W_L-w_L) (x_i)\vert \leq \frac{C}{\sqrt{\ve\theta}}(N^{-1}\ln N) \sqrt{\ve\theta} +CN^{-1}(\ln N)^2
\leq CN^{-1}(\ln N )^2.
\]

(iii) Let us now consider the error  $D^-(W_R-w_R)$ in the right fine mesh subregion $(1-\sigma _R, 1-\tau _R)$.
Using the truncation error bounds $\eqref{trunc_x}$ and the exponential bounds in Theorem  \ref{boundwrightleft} in the region $(1- \sigma_R,1)$, we have
\begin{eqnarray*}
 \vert \hat{L}^{N}D^-(W_R-w_R) (x_i)\vert&\leq& CN^{-1}\ln N \Bigl( \theta \sqrt{\frac{\theta }{\ve}} e^{-\rho _R(1-x_i)} + 1 \Bigr), \ x _i >1- \sigma_R +h_R;\\
 \vert \hat{L}^{N}D^-(W_R-w_R) (x_i)\vert
&\leq& C\theta \sqrt{\frac{\theta }{\ve}} (1+\sqrt{\frac{\theta }{\ve}}\frac{1 }{\ln N})e^{-\rho _R(1-x_i)}, \quad x _i = 1- \sigma_R+h_R.
\end{eqnarray*}
Consider the barrier function (which is a truncated version of  $\Psi _R$ defined in (\ref{bounddwr}))
\[
\Psi _3 (x_i) := (1+0.5\rho _R h_R)^{ -(N-i)}, 1- \sigma_R < x _i \leq 1; \qquad \Psi _3 (1- \sigma_R) :=0.
\]
For $x_i > 1-\sigma _R+h_R$, $\hat L^N\Psi _3 (x_i) \geq C \theta e^{-\frac{\rho _R}{2}(1-x_i)}$ and $\hat L^N\Psi _3(1-\sigma _R+h_R) \geq 0$.
This barrier function will be used to deal with the truncation error across the fine mesh region $(1-\sigma _R+h_R,1)$.
 An additional barrier function is required to manage the larger truncation error at $x_i=1-\sigma _R+h_R$.
 Consider the  step  barrier function
\[
\Psi _4 (1-\sigma _R) := 0 ; \qquad \Psi _4 (x_i) :=1,\quad  1-\sigma _R +h_R \leq x _i \leq 1.
\]
For $x_i > 1-\sigma _R +h_R$, $\hat L^N\Psi _4 (x_i) \geq 0$ and at the single point $1-\sigma _R +h_R$, using (\ref{hat-delta}),
\[
\hat L^N\Psi _4 (1-\sigma _R +h_R) \geq  \frac{\mu}{h_R} +\frac{\ve}{Hh_R} \qquad  (\geq \ C\frac{N\theta}{\ln N},\ \hbox{if} \ \theta >1).
\]
Then, in the particular case where $\theta > 1$, we deduce that
\[
\vert D^-(W_R-w_R) (x_i)\vert \leq \frac{C \sqrt{\theta}}{\sqrt{\ve}}(N^{-1}\ln N) \Psi _3 (x_i) +C\frac{1}{\theta N}\frac{\theta^2}{\ve}
e^{-\rho _R\sigma _R} \Psi _4 (x_i) + CN^{-1}\ln N.
\]
For $1-\sigma _R < x_i \leq 1- \tau _R$, we note that on the fine mesh
\[
e^{-\rho _R\sigma _R} \leq e^{-\rho _R\tau _R} \leq \rho _R ^2 \leq C(\frac{\ve}{\theta}) \quad \hbox{and} \quad (1+0.5\rho _R h_R)^{-1} \leq Ce^{-\frac{\rho _R\tau _R}{2}}. 
\]
Hence, 
\[
\vert D^-(W_R-w_R) (x_i)\vert \leq CN^{-1}\ln N, \quad \hbox{if} \quad \theta > 1. \]
When $\theta =1$, we employ an alternative barrier function to $\Psi _4(x_i)$ defined as
\[
\Psi _5(1-\sigma _R):=1, \qquad \Psi _5(x_i):=\frac{1-x_i}{\sigma _R-h_R}, \ 1-\sigma _R+h_R \leq  x_i \leq 1.
\]
Using (\ref{hat-delta}) and the fact that $\frac{b}{a} - \frac{\mu}{\sigma _R} >0$, we note that
\[
\hat L^N\Psi _5 (1-\sigma _R +h_R) \geq  C\frac{N}{(\ln N)^2}; \qquad\hat L^N\Psi _5 (x_i) \geq  0, \ x _i >1-\sigma _R +h_R.
\]Then, in the particular case where $\theta = 1$, we deduce that
\[
\vert D^-(W_R-w_R) (x_i)\vert \leq \frac{C \sqrt{\theta}}{\sqrt{\ve}}N^{-1}\ln N \Psi _3 (x_i) +CN^{-1}\frac{\ln N }{\ve }
e^{-\rho _R\sigma _R} \Psi _5 (x_i) + CN^{-1}\ln N.
\]
For $1-\sigma _R < x_i \leq 1- \tau _R$, we have $e^{-\rho _R(1-x_i)} \leq
e^{-\rho _R\tau _R} \leq C\ve$. Hence, 
\[
\vert D^-(W_R-w_R) (x_i)\vert \leq CN^{-1}\ln N + CN^{-1}\ln N \frac{\ve }{\ve }\leq CN^{-1}\ln N , \qquad \hbox{if} \ \theta = 1.
\]
(iv) We complete the argument, by dealing with the regular component.
 In the case of $\theta =1$, note the bound (\ref{react-diff-smooth}) for the regular component.
 Let us  consider the regular component in the case of $\theta  > 1$. Note that $\Vert L^N(v-V) \Vert \leq C(\ve + \mu)N^{-1}$ and so
\[
\Vert V-v \Vert \leq CN^{-1} \mu, \qquad \theta >  1.
\]
Note that we can confine the discussion to the mesh points in the region $(1-\sigma _R, 1-\tau _R)$.
Within the fine mesh region $(1-\sigma _R, 1)$, the error in the flux satisfies the first order problem
\[
-\ds\frac{\ve}{h_R} ({\mathscr V}^-_{i+1}-{\mathscr V}^-_{i})+\mu a(x_i) {\mathscr V}^-_{i}=\hat {\mathscr T}_{i},
\quad \vert {\mathscr V}^-_{N} \vert \leq C\frac{\mu}{\ve}N^{-1};
\]
where
$
\hat {\mathscr T}_{i}:=L^{N}(V-v)(x_i)-b(x_i)(V-v)(x_i).
$
Note further that
$$
\Vert \hat {\mathscr T}\Vert \le CN^{-1} (\ve +\mu)\le C\mu N^{-1}.
$$
Thus, with $\rho :=\ds\frac{\alpha \mu h_R}{\ve} \leq CN^{-1} \ln N$, we have
$$
\vert {\mathscr V}^-_{i}\vert =\bigl(1+\ds\frac{\mu h_R}{\vr}a(x_i) \bigr)^{-1}  \bigl\vert \ds\frac{h_R}{\ve}\hat {\mathscr T}_{i}+
 {\mathscr V}^-_{i+1}\bigr\vert \le C(1+\rho )^{-1} \bigl( \ds\frac{\rho}{\mu} \Vert \hat {\mathscr T} \Vert + \vert {\mathscr V}^-_{i+1} \vert \bigr).
$$
We have the following estimate at $x_i$ (within the fine mesh where $(1+\rho)^{-1} \leq C e^{-\rho/2}$ for $N$ sufficiently large)
\begin{eqnarray*}
\vert {\mathscr V}^-_{i} \vert
 &\le&  (1+\rho )^{-1} \ds\frac{\rho}{\mu} \Vert \hat {\mathscr T} \Vert \ds\frac{1-(1+\rho)^{-(N-i)}}{1-(1+\rho)^{-1}}+C(1+\rho)^{-(N-i)}   \vert {\mathscr V}^-_{N} \vert \\
&\le& C N^{-1} + C\frac{\mu N^{-1}}{\ve}(1+\rho)^{-(N-i)} \leq C N^{-1} + C\frac{\mu N^{-1}}{\ve}e^{-\frac{\alpha \mu}{2\ve}(1-x_i)}.
\end{eqnarray*}
Hence, for $1-\sigma _R < x_i \leq 1 -\tau _R$,
$
\vert {\mathscr V}^-_{i} \vert \leq CN^{-1}$.
  \end{proof}

\section{Global error bounds}

In this section, we examine the global accuracy of the linear interpolant
\[
\bar U (x):= \sum _{i=1}^{N-1} U(x_i) \phi _i(x),\quad x \in [0,1],
\]
where $\phi _i(x)$ is the standard  piecewise linear basis functions, defined by the nodal values of $\phi_i(x_k)= \delta _{i,k}$.
 Note that
\[
(\bar u -u)' (x) =  D^-u(x_i)-u'(x), \quad x \in (x_{i-1},x_i];
\] and, hence, we have the following bound on the linear interpolant $\bar g$ (for any $g \in C^1$)  in the subinterval $I_i:= (x_{i-1},x_i)$
\begin{subequations}\label{interpol-error}
\begin{eqnarray}
\Vert g-\bar g \Vert _{I_{i}} &\leq& C\min \{ h^2_i \Vert g'' \Vert _{I_{i}},  \int _{t=x_{i-1}}^{x_i} \vert g'(t) \vert dt \} \\
\Vert (g-\bar g)' \Vert _{ I_{i}}
&\leq& C\min \{ h_i \Vert g'' \Vert _{ I_{i}}, \Vert g' \Vert _{ I_{i}} \}.
\end{eqnarray}
\end{subequations}

\begin{theorem}\label{Inter}
We have the interpolation error bound
\begin{eqnarray*}
\Vert u- \bar u \Vert _{1,\chi} \leq C N^{-1}\ln N,
\end{eqnarray*}
where $u$ is the solution of (\ref{cp1}) and $\bar u$ is the piecewise linear interpolant of $u$.
\end{theorem}
\begin{proof}
Using the decomposition $u=v+(u-v)(0)w_L+(u-v-w_L)(1)w_R$, splitting the argument to inside and outside the computational layer regions
 $[0,\sigma _L],[1-\sigma _R, 1] $, using  the bounds (\ref{interpol-error}a), (\ref{exp-bound-w}b) and (\ref{exp-bound-w}c), we have the following interpolation error
\begin{equation}\label{interC0}
\Vert u -\bar u \Vert  \leq C (N^{-1}\ln N)^2.
\end{equation}
We next want to estimate the global error in approximating the scaled flux.
For the regular component it trivially follows that
\[
\Vert (v-\bar v)' \Vert _{ I_{i}} \leq CN^{-1}.
\]
For the left layer component, we first consider the case where $\tau _L \leq \sigma _L$. By using   the bound (\ref{interpol-error}b), we can obtain
\begin{eqnarray*}
\sqrt{\ve \theta} \Vert (w_L-\bar w_L)' \Vert _{ I_{i}} \leq CN^{-1} \ln N,\quad \hbox{for} \quad x_i \leq \tau _L\\
 \Vert (w_L-\bar w_L)' \Vert _{ I_{i}}\leq C\frac{h_L}{\ve \theta} e^{-\frac{\sqrt{\gamma \alpha}}{2\sqrt{\ve \theta }}\tau _L}\leq CN^{-1}\ln N,\quad \hbox{for} \quad \tau _L < x_i \leq \sigma _L \\
\Vert (w_L-\bar w_L)' \Vert _{ I_{i}}\leq \frac{C}{\sqrt{\ve \theta }} e^{-\frac{\sqrt{\gamma \alpha}}{4\sqrt{\ve \theta }}\sigma _L}e^{-\frac{\sqrt{\gamma \alpha}}{4\sqrt{\ve \theta }}\tau _L} \leq CN^{-1} ,\quad \hbox{for} \quad x_i > \sigma _L.
\end{eqnarray*}
For the alternative case, where $  \sigma _L \leq \tau _L$ we have the bounds
\begin{eqnarray*}
\sqrt{\ve \theta} \Vert (w_L-\bar w_L)' \Vert _{ I_{i}} \leq CN^{-1} \ln N,\quad \hbox{for} \quad x_i \leq \sigma _L\\
 \sqrt{\ve \theta} \Vert (w_L-\bar w_L)' \Vert _{ I_{i}}\leq C e^{-\frac{\sqrt{\gamma \alpha}}{2\sqrt{\ve \theta }}\sigma _L}\leq CN^{-2},\quad \hbox{for} \quad \sigma _L < x_i \leq \tau _L \\
\Vert (w_L-\bar w_L)' \Vert _{ I_{i}}\leq \frac{C}{\sqrt{\ve \theta }}e^{-\frac{\sqrt{\gamma \alpha}}{4\sqrt{\ve \theta }}\tau _L} e^{-\frac{\sqrt{\gamma \alpha}}{4\sqrt{\ve \theta }}\sigma _L} \leq CN^{-1} ,\quad \hbox{for} \quad x_i > \tau _L.
\end{eqnarray*}
A similar argument is used for the right layer component. We begin with  the case of $\tau _R \leq \sigma _R$:
\begin{eqnarray*}
\sqrt{\frac{\ve}{ \theta}} \Vert (w_R-\bar w_R)' \Vert _{ I_{i}} \leq CN^{-1} \ln N,\quad \hbox{for} \quad x_i > 1- \tau _R\\
\Vert (w_R-\bar w_R)' \Vert _{ I_{i}}\leq Ch_R{{\frac{\theta}{\ve}}} e^{-\frac{\sqrt{\gamma \alpha \theta}}{\sqrt{\ve }}\tau _R}\leq CN^{-1} \ln N,\quad \hbox{for} \quad 1-\sigma _R < x_i \leq 1-\tau _R \\
\Vert (w_R-\bar w_R)' \Vert _{ I_{i}}\leq C\sqrt{\frac{ \theta}{\ve}} e^{-\frac{\sqrt{\gamma \alpha \theta}}{2\sqrt{\ve }}\tau _R}e^{-\frac{\sqrt{\gamma \alpha \theta}}{2\sqrt{\ve }}\sigma _R} \leq CN^{-2} ,\quad \hbox{for} \quad x_i \leq 1- \sigma _R.
\end{eqnarray*}
For the alternative case, where $  \sigma _R \leq \tau _R$ we have the bounds
\begin{eqnarray*}
\sqrt{\frac{\ve}{ \theta}} \Vert (w_R-\bar w_R)' \Vert _{ I_{i}} \leq CN^{-1} \ln N,\quad \hbox{for} \quad x_i > 1-\sigma _R\\
\sqrt{\frac{\ve}{ \theta}} \Vert (w_R-\bar w_R)' \Vert _{ I_{i}}\leq C e^{-\frac{\sqrt{\gamma \alpha \theta}}{\sqrt{\ve }}\sigma _R}\leq CN^{-2},\quad \hbox{for} \quad 1-\tau _R < x_i \leq 1-\sigma _R \\
\Vert (w_R-\bar w_R)' \Vert _{ I_{i}}\leq C\sqrt{\frac{ \theta}{\ve}}e^{-\frac{\sqrt{\gamma \alpha \theta}}{2\sqrt{\ve }}\tau _R} e^{-\frac{\sqrt{\gamma \alpha \theta}}{2\sqrt{\ve }}\sigma _R} \leq CN^{-1} ,\quad \hbox{for} \quad x_i < 1-\tau _R.
\end{eqnarray*}
\end{proof}

We conclude with the statement of the main result of this paper.
\begin{theorem}\label{Main}
We have the global error  bound
\begin{eqnarray*}
\Vert u- \bar U \Vert _{1,\chi} &\leq& C N^{-1}\ln N , \quad \hbox{assuming} \quad (\ref{assumption}) \\
\Vert u- \bar U \Vert _{1,\chi} &\leq& C N^{-1}(\ln N)^3;
\end{eqnarray*}
where $u$ is the solution of (\ref{cp1}) and $U$ is the solution of (\ref{upwindop}).
\end{theorem}
\begin{proof}
(i) Assume first that (\ref{assumption}) applies.
Combining the interpolation bound (\ref{interC0}) with the nodal error bound (\ref{NodalError}), we arrive at the following global error estimate:
\[
\Vert u -\bar U \Vert _{1,\chi} \leq C N^{-1}\ln N.
\]
 Note also that
$
(\bar U -\bar u)'(x) =  D^-(U-u)(x_i),\quad \forall x \in (x_{i-1},x_i].
$
Use this bound, Theorem \ref{NodalBound} and the interpolation bound in Theorem \ref{Inter} to finish.

(ii) If $\sigma _R =1/4$  then
\begin{eqnarray*}
\ve \vert u \vert _4 + \mu \vert u \vert _3 &\leq& C  \ve ^{-1} \leq C (\ln N )^2; \quad \hbox{if} \quad \theta =1; \\
\ve \vert u \vert _4 + \mu \vert u \vert _3 &\leq& C  \mu ^4 \ve ^{-3} \leq C \mu (\ln N )^3; \quad \hbox{if} \quad \theta > 1.
\end{eqnarray*}
If $\sigma _R  =1/4$, note that the mesh is uniform and apply the argument used to bound $\Vert v -\bar V \Vert _{1,\chi}$ to the entire solution.
If $\sigma _R <\sigma _L =1/4$, then combine the analysis for $v+W_L$ together as for the regular component and treat the error $\Vert w_R -\bar W_R \Vert _{1,\chi}$ as before.
\end{proof}

\section{Numerical results}

Consider the following  constant coefficient sample problem
\begin{equation}\label{eg1}
-\ve u''+\mu u' + u=x, \qquad\  x \in (0,1);\quad  u(0)=1, \ u(1)=0.
\end{equation}
Letting $m_1 := \mu+\sqrt{\mu^2+4\ve}$ and $m_2 := \mu-\sqrt{\mu^2+4\ve}$, the exact solution is given by
\[u(x) = \left(\frac{(1+\mu)e^{\frac{m_2}{2\ve}}+1-\mu}{e^{\frac{(-m_1+m_2)}{2\ve}}-1}\right)e^{-\frac{m_1(1-x)}{2\ve}}+\left(\frac{(\mu-1)e^{\frac{-m_1}{2\ve}}-1-\mu}{e^{\frac{-m_1+m_2}{2\ve}}-1}\right)e^{\frac{m_2 x}{2\ve}}+x-\mu.
\]
\begin{figure}[htb]
\centering
\begin{subfigure}{.5\textwidth}
  \centering
  \includegraphics[width=.8\linewidth]{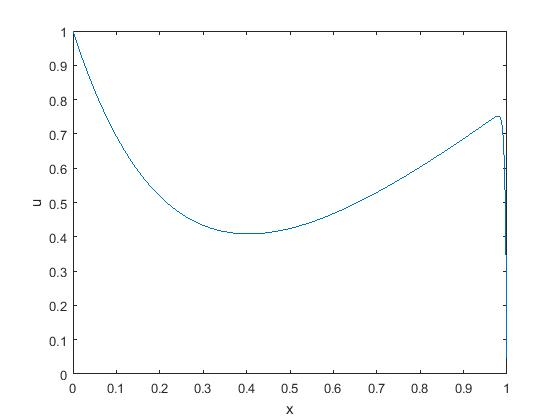}
  \caption{ $\ve = 2^{-10}$ and $\mu = 2^{-2}$}
  \label{fig:sol1}
\end{subfigure}%
\begin{subfigure}{.5\textwidth}
  \centering
  \includegraphics[width=.8\linewidth]{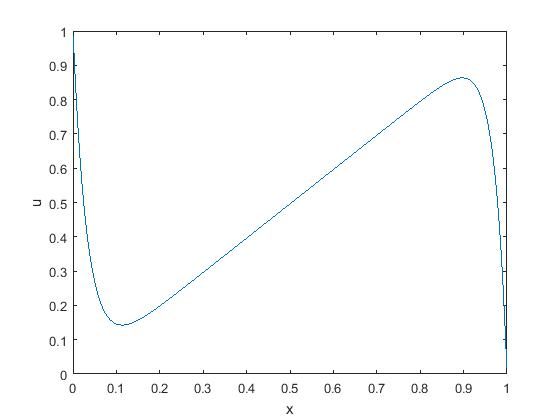}
  \caption{$\ve = 2^{-10}$ and $\mu = 2^{-8}$}
  \label{fig:sub2}
\end{subfigure}
\caption{Solution of \eqref{eg1} in the case of $\theta  > 1$ (left) and $\theta= 1$ (right).}
\label{fig:test}
\end{figure}
A sample plot of the solution in the convection-dominated case and in the reaction-dominated case are displayed in Figure 1. 

The solution to this problem was approximated by applying the upwind finite difference \eqref{upwindop} on the piecewise-uniform Shishkin mesh defined  in \eqref{scheme1}. Numerical approximations $U$ to the solution $u$ of \eqref{eg1} were generated over the parameter sets $S_\ve := \{ 2^{-2j}; j=0,1,\ldots, 20 \}, S_\mu = \{ 2^{-2j}; j=0,1,\ldots, 10 \}$ and  $N = \{ 2^{k}; k=6,7,\ldots, 11\} $. For each set of parameters, a global approximation $\bar U$ (to the solution $u$ of \eqref{eg1}) was generated using  linear interpolation.  
For each particular triple $(\ve, \mu, N)$ set of parameter values, the global scaled $C^1$ error $\Vert u - \bar U \Vert_{1,\chi}$ (as defined in \eqref{norm}) is estimated by calculating 
\[
 E^N_{\ve, \mu}  := \Vert \chi (u' - \bar U' )\Vert_{\Omega_{fine}} +  \Vert u - \bar U \Vert_{\Omega_{fine}},
\]
where $\Omega_{fine}$ is  a fine Shishkin mesh \eqref{scheme1},\eqref{sigma1} with $N = 8192$. 
The results presented in Tables 1 and 2 display parameter-uniform convergence in the $\Vert \cdot\Vert_{1,\chi}$ norm. 

 For each $N$, the parameter-uniform   orders of global convergence $p^N$ are estimated by computing
\[
 E^{N}:=\max_{\ve \in S_\ve , \mu \in S_\mu } E^{N}_{\ve, \mu} , \quad p^{N}:=\log _2(E^{N}/E^{2N}),
\]
which are displayed in Table 3. For the particular test problem (\ref{eg1}), these parameter-uniform   orders of global convergence are higher than the theoretical rates established in Theorem 12. 

\begin{table}[ht]
\caption{Computed global errors $E^{N}_{\ve, 2^{-4}}$,  where  $\mu = 2^{-4}$ and $\ve$ varies}
\center{
\begin{tabular}{|c|c|c|c|c|c|c|}  \hline 
$\epsilon$ / N&64&128&256&512&1024&2048\\ \hline
$2^0$&8.30e-03&4.12e-03&2.07e-03&1.04e-03&5.19e-04&2.60e-04\\ \hline
$2^{-2}$&2.86e-02&1.43e-02&7.25e-03&3.65e-03&1.83e-03&9.17e-04\\ \hline
$2^{-4}$&4.19e-02&2.09e-02&1.02e-02&4.81e-03&2.19e-03&1.09e-03\\ \hline
$2^{-6}$&1.17e-01&5.95e-02&2.95e-02&1.41e-02&6.33e-03&2.49e-03\\ \hline
$2^{-8}$&3.81e-01&2.08e-01&1.07e-01&5.29e-02&2.45e-02&9.95e-03\\ \hline
$2^{-10}$&{\bf 7.23e-01}&{\bf 4.52e-01}&2.59e-01&1.37e-01&6.71e-02&2.93e-02\\ \hline
$2^{-12}$&6.40e-01&4.51e-01&{\bf 2.92e-01}&{\bf 1.75e-01}&{\bf 9.64e-02}&{\bf 4.73e-02}\\ \hline
$2^{-14}$&6.19e-01&4.38e-01&2.84e-01&1.71e-01&9.43e-02&4.67e-02\\ \hline
$2^{-16}$&6.14e-01&4.35e-01&2.82e-01&1.69e-01&9.38e-02&4.66e-02\\ \hline
$2^{-18}$&6.12e-01&4.34e-01&2.82e-01&1.69e-01&9.36e-02&4.65e-02\\ \hline
.&.&.&.&.&.&.\\ 
.&.&.&.&.&.&.\\ 
.&.&.&.&.&.&.\\ 
.&.&.&.&.&.&.\\ \hline
$2^{-40}$&6.12e-01&4.34e-01&2.81e-01&1.69e-01&9.36e-02&4.65e-02\\ \hline
\end{tabular}
}
\end{table}

\begin{table}[ht]
\caption{Computed maximum global errors $E^{N}_{\mu}$,  in the scaled $C^1$ norm, measured over the set $S=\{ \ve =2^{-2j}, j=0,1..,20\}$ for various values of $\mu$}
\center{\begin{tabular}{|c|c|c|c|c|c|c|}  \hline 
$\mu$ / N&64&128&256&512&1024&2048\\ \hline
$2^0$& 4.83e-01&3.39e-01&2.20e-01&1.32e-01&7.33e-02&   3.64e-02 \\ \hline
$2^{-2}$&5.24e-01&3.60e-01&2.34e-01&1.40e-01& 7.75e-02&3.84e-02 \\ \hline
$2^{-4}$& 7.23e-01& 4.52e-01& 2.92e-01& 1.75e-01&   9.64e-02& 4.73e-02\\ \hline
$2^{-6}$& 1.07e+00& 6.72e-01& 3.86e-01& 2.05e-01&   1.11e-01& 5.30e-02 \\ \hline
$2^{-8}$ & 1.09e+00  &7.59e-01& 4.88e-01&   2.88e-01&   1.53e-01&   6.87e-02 \\ \hline
 $2^{-10}$& 1.09e+00 & 7.61e-01&   4.89e-01&   2.89e-01&   1.54e-01&   6.89e-02 \\ \hline
 $2^{-12}$ &  1.09e+00 &  7.62e-01 &  4.90e-01 &   2.90e-01&   1.54e-01&   6.89e-02\\ \hline
$2^{-14}$ &  1.09e+00&   7.62e-01 &  4.90e-01&   2.90e-01 &  1.54e-01 &  6.89e-02 \\ \hline
.&.&.&.&.&.&.\\ 
.&.&.&.&.&.&.\\ \hline
$2^{-20}$  & 1.09e+00 &  7.62e-01 &  4.90e-01&   2.90e-01 &  1.54e-01 &  6.89e-02 \\ \hline
\end{tabular}
}\end{table}

\begin{table}[ht]
\caption{Computed orders of parameter-uniform convergence in the scaled $C^1$ norm $\Vert \cdot\Vert_{1,\chi}$}
\center{
\begin{tabular}{|c|c|c|c|c|c|}  \hline 
$N $  &64&128&256&512&1024\\ \hline
$p ^N$&0.52&0.64&0.76&0.91&1.16 \\ \hline
\end{tabular}
}
\end{table}


\noindent {\bf References}


\vskip0.5cm 

{\bf Appendix A. Proof of Theorem 2.}

\begin{proof} (i)  The argument follows \cite{highorder2parameter} closely. We first consider the reaction-diffusion case, where $\theta =1$. We decompose the regular component, as in \cite{highorder2parameter}, in a series of terms of increasing half powers of $\ve$. That is, let
\begin{eqnarray*}
&&v= \sum _{i=0}^3 \ve ^{i/2} v_i,\quad
\hbox{where} \quad L_0v_0=f; \quad \sqrt{\ve} L_0v_i = (L_0-L_{\ve,\mu})v_{i-1}, \ i= 1,2; \\
&&\hbox{and} \qquad \sqrt{\ve} L_{\ve,\mu}v_3 = (L_0-L_{\ve,\mu})v_{2}, \qquad v_3(0)=v_3(1)=0.
\end{eqnarray*}
Assuming  $a \in C^7(\Omega), b,f \in C^9(\Omega)$, which is more regularity to that assumed  in \cite{highorder2parameter}, we see that
\begin{eqnarray*}
&& v_i \in C^{9-2i}(\Omega), i=0,1,2; \quad v_3 \in C^5(\Omega) \\
&& \hbox{and} \quad
\vert v \vert _k \leq C(1+\bigl(\sqrt{\ve}\bigr)^{3-k}), \quad \textrm{for}
\  0\leq k \leq 5, \quad \hbox{and} \quad \theta =1.
\end{eqnarray*}
For the convection-diffusion case, where $\theta >1$, we again follow \cite{highorder2parameter} and decompose the regular component in a series of terms of increasing integer powers of $\ve$ as follows:
Define
\begin{eqnarray*}
&&v= \sum _{i=0}^3 \ve ^{i} v_i,\quad
\hbox{where} \quad L_\mu v_0=f; \quad \sqrt{\ve} L_\mu v_i = (L_\mu-L_{\ve,\mu})v_{i-1}, \ i= 1,2; \\
&&\hbox{and} \qquad \sqrt{\ve} L_{\ve,\mu}v_3 = (L_\mu-L_{\ve,\mu})v_{2}, \qquad v_2(0)= v_3(0)=v_3(1)=0
\end{eqnarray*}
and $v_0(0),v_1(0)$ are suitably chosen. Assuming  $a, b,f \in C^6(\Omega)$, then following \cite{highorder2parameter} to identify  appropriate choices for $v_0(0),v_1(0)$, we deduce that
\begin{eqnarray*}
&& v_0 \in C^7(\Omega), v_1 \in C^6(\Omega), \quad v_2, v_3 \in C^5(\Omega);\\
&&
\vert v _i \vert _k \leq C(1+\mu ^{3-2i-k}), \quad \textrm{for}
\  0\leq k \leq 7-i, \quad \hbox{and} \quad i=0,1,2; \\
&& \vert v_3 \vert _k \leq C\bigl( \frac{\mu}{\ve} \bigr)^k \mu ^{-3}, \qquad 0 \leq k \leq 5.
\end{eqnarray*}
From these bounds we deduce that
\[
\vert v \vert _k \leq C(1+ \mu ^{3-k}\sum _{j=0}^3 (\ve \mu ^{-2})^j\ ) \leq C(1+\mu ^{3-k}), \quad 0 \leq k \leq 5.
\]
In other words,
\begin{eqnarray*}
\vert v \vert _4 &\leq& C(1+\frac{\mu}{\ve}\sum _{j=1}^4 (\ve \mu ^{-2})^j) \leq C(1+\frac{\mu}{\ve} ); \\
 \vert v \vert _5 &\leq& C(1+(\frac{\mu}{\ve})^2\sum _{j=2}^5 (\ve \mu ^{-2})^j) \leq C(1+(\frac{\mu}{\ve})^2 ).
\end{eqnarray*}
All of the bounds (\ref{boundvf21}) have now been established in both cases of $\theta =1$ and $\theta >1$.

(ii) We next establish the pointwise bounds on the layer components, using a comparison principle.
Observe that
\begin{eqnarray*}
L_{\ve,\mu} e^{-\frac{\sqrt{\gamma \alpha}}{2\sqrt{\ve \theta }}x} = a(\frac{b}{a} - \frac{1}{a}\frac{\gamma \alpha}{4\theta} - \frac{\mu}{2} \sqrt{\frac{\gamma \alpha}{\ve \theta}} ) e^{-\frac{\sqrt{\gamma \alpha}}{2\sqrt{\ve \theta }}x} \\
\geq a \gamma (1 - \frac{1}{4\theta} - \frac{1}{2} \sqrt{\frac{\mu ^2 \alpha}{\gamma \ve \theta}} ) e^{-\frac{\sqrt{\gamma \alpha}}{2\sqrt{\ve \theta }}x} \geq 0;
\end{eqnarray*}
and
\begin{eqnarray*}
L_{\ve,\mu} e^{-\frac{\sqrt{\gamma \alpha \theta}}{\sqrt{\ve }}(1-x)} = a(\frac{b}{a} - \frac{1}{a}\gamma \alpha \theta + \mu \sqrt{\frac{\gamma \alpha \theta}{\ve }})  e^{-\frac{\sqrt{\gamma \alpha \theta}}{\sqrt{\ve }}(1-x)}\\
\geq a(\frac{b}{a} - \gamma  \theta + \mu \sqrt{\frac{\gamma \alpha \theta}{\ve }} )  e^{-\frac{\sqrt{\gamma \alpha \theta}}{\sqrt{\ve }}(1-x)}
 \geq 0.
\end{eqnarray*}
The comparison principle then yields the pointwise bounds (\ref{exp-bound-w}a).

(iii) From the bounds (\ref{crude}) established in Lemma 1, we deduce the following derivative bounds on the singular components $w_L,w_R$. For $0\leq k \leq 5$,
\begin{equation}
\vert w_L \vert_k,  |w_{R}|_k \leq C \bigl(\sqrt{\frac{\theta}{\ve }} \bigr)^k.
\end{equation}
When $\theta >1$, we can derive sharper bounds on the derivatives of  $w_{L}$ by introducing the secondary decomposition
\begin{subequations}\label{sec-decomp}
\begin{eqnarray}
&&w_L= \sum _{i=0}^3 \ve ^{i} w_i, \quad \hbox{where} \quad L_\mu w_0=0, \ w_0(0)=1;\\ &&
 \quad \ve L_\mu w_i = (L_\mu-L_{\ve,\mu})w_{i-1}, \quad w_i(0)=0,  i= 1,2; \ \\
&&\hbox{and} \qquad \ve L_{\ve,\mu}w_3 = (L_\mu-L_{\ve,\mu})w_{2}, \qquad  w_3(0)=w_3(1)=0.
\end{eqnarray}\end{subequations}
Observe that $w_L(1)=w_0(1)+\ve w_1(1) +\ve ^2w_2(1)  \neq 0$.
From this expansion one can deduce that
\begin{equation*}
\vert w_L(1) \vert \leq e^{-\frac{\gamma}{\mu}}, \qquad
\vert w_{L}\vert _k\leq
C\mu^{-k}, \ 1\leq k \leq 5.
\end{equation*}
Hence, we have deduced that
\[
 |w_{L}|_k \leq C\bigl(\frac{1}{\sqrt{\ve \theta }}\bigr)^k,\quad 1\leq k \leq 5  .
\]
Our next  step in the proof is to deduce (\ref{exp-bound-w}b), which are pointwise exponential bounds on the derivatives of the layer component $w_R$. For those points within the right layer, we have that
\[
e^{-\frac{\sqrt{\gamma \alpha \theta}}{\sqrt{\ve }}(1-x)} \geq C, \ 1- \sqrt{\frac{\ve }{\theta }}\leq x \leq 1
\]
and so
\[
 |w_{R}(x)|_k \leq C\bigl(\sqrt{\frac{\theta}{\ve }} \bigr)^k\leq C\bigl(\sqrt{\frac{\theta}{\ve }} \bigr)^k e^{-\frac{\sqrt{\gamma \alpha \theta}}{\sqrt{\ve }}(1-x)}, \ 1- \sqrt{\frac{\ve }{\theta }}\leq x \leq 1.
 \]
 Now we return to the argument from Lemma \ref{apriori}. If $x < 1- \sqrt{\frac{\ve }{\theta }}$, construct a neighbourhood $N_x=(p-r,p)$ so that $x \in N_x$. Then there exists a $y \in N_x$ such that
 \[
\vert w_R'(y) \vert  \leq \frac{2 \Vert w_R \Vert _{N_x}}{r}.
\]
\begin{eqnarray*}
w_R'(x) &=&w_R'(y) + \int _{t=y}^x w_R'' \ d t \
=w_R'(y) + \frac{1}{\ve} \int _{t=y}^x \mu aw_R'+bw_R \ d t \\
&=&w_R'(y) + \frac{\mu}{\ve} ((aw_R)(x)-(aw_R)(y))  -\frac{1}{\ve}\int _{t=y}^x \mu a'w_R-bw_R \ d t. \end{eqnarray*}
Thus
\[
\vert w_R'(x) \vert \leq C(\frac{1}{r}+ \frac{\mu}{\ve}  +\frac{r}{\ve}) \Vert w_R \Vert _{N_x} \leq C(\frac{1}{r}+ \frac{\mu}{\ve}  +\frac{r}{\ve})e^{-\frac{\sqrt{\gamma \alpha \theta}}{\sqrt{\ve }}(1-x)}e^{\frac{\sqrt{\gamma \alpha \theta}}{\sqrt{\ve }}r} .
\]
By taking
\[
r = \sqrt{\frac{\ve}{2\theta}}
\]
we deduce
\[
\vert w_R'(x) \vert \leq C\sqrt{\frac{\theta}{\ve}}e^{-\frac{\sqrt{\gamma \alpha \theta}}{\sqrt{\ve }}(1-x)}.
\]
 From the differential equation that defines $w_R$, we have that
\[
\ve \vert w_R''(x)  \vert \leq C \sqrt{\theta \ve} \vert w_R'(x) \vert + C \vert w_R(x) \vert.
\]
which will establish the  bound on the second order derivative of $w_R$.
Use the bounds in Lemma 1, to establish the bounds (\ref{exp-bound-w}b) on the higher derivatives of $w_R(x)$.

(iv) To complete the proof, we  establish the bound (\ref{exp-bound-w}c). For the case of $\theta =1$, the above argument (used to establish (\ref{exp-bound-w}b)) can be repeated (with $1-x$ replaced by $x$). In the other case of $\theta >1$, we use the decomposition (\ref{sec-decomp}). Observe that for $0\leq k \leq 5$,
\[
\vert w_i(x) \vert _k \leq \frac{C}{\mu ^{2i+k}} e^{-\gamma x/\mu}, \qquad i=0,1,2
\]
and hence, using a maximum principle for the second order operator $L_{\ve,\mu}$ we have
\[
\vert w_3(x) \vert  \leq \frac{C}{\mu ^{6}} e^{-\frac{\gamma x}{2\mu}}.
\]
Now  repeat the argument used to establish (\ref{exp-bound-w}b) (taking $r=\mu$) to deduce that for $x > \mu$,
\[
\vert w'_3(x) \vert  \leq \frac{C}{\mu ^{7}} (1+\theta)e^{-\frac{\gamma x}{2\mu}} \quad \hbox{and} \quad \vert w''_3(x) \vert  \leq \frac{C}{\mu ^{8}} (1+\theta +\theta ^2)e^{-\frac{\gamma x}{2\mu}} .
\]
Hence, since we are in the case of $\ve \leq C \mu ^2$,
\begin{eqnarray*}
\vert w'_L(x) \vert  &\leq&\frac{C}{\mu }(1+ \frac{\ve}{\mu^2}  +\frac{\ve ^2}{\mu^4} + \frac{\ve ^3}{\mu^6}(1+\frac{\mu ^2}{\ve})) e^{-\frac{\gamma x}{2\mu}} \leq \frac{C}{\mu }e^{-\frac{\gamma x}{2\mu}}, \\ \vert w_L(x) \vert _k &\leq& \frac{C}{\mu ^k}e^{-\frac{\gamma x}{2\mu}},\quad k=2,3.
\end{eqnarray*}
Continuing this argument for the higher derivatives establishes (\ref{exp-bound-w}c) for $\theta >1$.
\end{proof}

\vskip0.25cm 

{\bf Appendix B. Proof of Lemma 8. }

\begin{proof}
Using the bounds (\ref{boundvf21}) on the derivatives of the regular component $v$, we have the truncation error bound
\[
|L^{N}(V-v)(x_i)|:=  \left\{
\begin{array}{lll}
C(\ve + \mu) N^{-1}, & \hbox{if } x_i=\sigma _L,1-\sigma_R, \\
C(\sqrt{\ve \theta} N^{-1} + \mu) N^{-1}&\hbox{otherwise}
\end{array} .
\right.\]
(i) Looking first to establish a  bound at the end-point $x = 0$, if $\theta>1$, consider the linear barrier function
\[B(x_i):=  C_1(\frac{\ve}{\mu} + 1) N^{-1}x_i. \]
 Observe that $L^N(B(x_i) \pm (V-v)(x_i)) \geq 0 $ for $C_1$ large enough. Applying the discrete minimum  principle and using $\theta = \frac{\alpha\mu^2}{\gamma \ve}>1$ we deduce that
\[|(V-v)(x_i)| \leq CN^{-1}(\frac{\ve}{\mu})x_i \leq CN^{-1}\mu x_i,
\]
yielding the bound
$
\vert D^+(V-v)(0) \vert \leq C\mu N^{-1}.
$

\noindent (ii) In the reaction-diffusion case (where $\theta = 1$) consider the barrier function
\[
B_2(x_i) := C_1(\sqrt{\ve}N^{-2}\ln N R_1(x_i))+C_2(N^{-2}\frac{\sqrt{\ve}}{\beta}+N^{-1}x_i),
\]
where the wedge function $R_1(x_i)$ is defined by
\begin{equation}\label{wedge}
R_1(x_i) :=  \left\{
\begin{array}{ll}
\frac{x_i}{\sigma_L}, & \hbox{if } \  x_i \leq \sigma_L, \\
1, & \hbox{if }\ \sigma_L< x_i  < 1-\sigma_R ,\\
\frac{1-x_i}{\sigma_R}, & \hbox{if }\   x_i \geq  1-\sigma_R.
\end{array}
\right.  
\end{equation}
We find that
\begin{equation}\label{lnr1}
L^{N} R_1(x_i) \geq  \left\{
\begin{array}{lll}
0, & \hbox{if } \ x_i < \sigma_L , \\
\frac{\ve N}{\sigma_L} +\frac{\mu \alpha }{\sigma _L}, & \hbox{if} \ x_i = \sigma_L\\
0, & \hbox{if }\ \sigma_L<x_i < \sigma_R , \\\frac{\ve N}{\sigma_R}, &\hbox{if}\ x_i = 1-\sigma_R, \\
-\frac{\mu a}{\sigma_R}, & \hbox{if }\ x_i > 1-\sigma_R .
\end{array}
\right.
\end{equation}
Since
\[
L^N(N^{-2}\frac{\sqrt{\ve}}{\beta}+N^{-1}x_i) \geq a\mu N^{-1} +N^{-2}\sqrt{\ve}+N^{-1}x_ib \geq CN^{-1}(\mu +N^{-1}\sqrt{\ve}),
\]
we see that
\[
L^{N} B_2(x_i) \geq  \left\{
\begin{array}{lll}
CN^{-1}(\mu +N^{-1}\sqrt{\ve}),  &\hbox{if } x_i < \sigma_L, \sigma_L<x_i < \sigma_R\\
CN^{-1}(\sqrt{\ve}N^{-1}\ln N (\frac{\ve N}{\sigma_L})+\mu ), & \hbox{if } x_i = \sigma_L\\
CN^{-1}(\sqrt{\ve}N^{-1}\ln N (\frac{\ve N}{\sigma_R})+\mu ), &\hbox{if } x_i = 1-\sigma_R, \\
C_1N^{-1}(\sqrt{\ve}N^{-1}\ln N (-\frac{\mu a}{\sigma_R}))+C_2N^{-1}(\mu +N^{-1}\sqrt{\ve}), & \hbox{if } x_i > 1-\sigma_R .
\end{array}
\right.
\]
Now when $\theta = 1$, for the bound at the transitions points, note that
\[\sqrt{\ve}N^{-2}\ln N (\frac{\ve N}{\sigma_L}) = \sqrt{\ve}N^{-2}\ln N (\frac{\ve N}{\sigma_R}) = \frac{N^{-1}\ve\sqrt{\gamma \alpha}}{4}.
\]
Also for $C_2$ sufficiently large, for the bound in the layer region near $x=1$,
\[
C_1(\sqrt{\ve}N^{-2}\ln N (-\frac{\mu a}{\sigma_R})+C_2(\mu N^{-1})\geq C_3\mu N^{-1}.
\]
We therefore have deduced that
\[|(V-v)(x_i)| \leq B_2(x_i) = C_1(\sqrt{\ve}N^{-2}\ln N)R_1(x_i)+C_2(N^{-2}\frac{\sqrt{\ve}}{\beta}+N^{-1}x_i).
\]
Using $\sigma_L = C\sqrt{\ve}\ln N$ we see that
\[|(V-v)(h_L)| \leq CN^{-1}(N^{-1}h_L+N^{-1}\sqrt{\ve}+h_L),
\]
which yields the bound
\[
\vert{ \cal V}^+_{0} \vert \leq CN^{-1}(N^{-1}+N^{-1}\frac{\sqrt{\ve}}{h_L}+1) \leq CN^{-1}.
\]
Hence, for both cases, we have established the bound at the left end-point $x=0$.

 (iii) For the other end of the interval with $x=1$, consider the case of  $\theta>1$ and   the barrier  function
\[
B_3(x_i):=  C_1(\frac{\ve}{\mu}+1)N^{-1}(x_i-1+\tilde \psi(x_i)),
\]
where the mesh function $\tilde \psi(x_i)$ satisfies
\begin{equation}\label{eqnpsi}
-\ve \delta^2 \tilde  \psi +C_{*}\sqrt{\ve\theta} D^- \tilde \psi=0, \ x_i\in (0,1), \
\tilde \psi(0)=1, \ \tilde \psi(1)=0; \ C_{*} := A\sqrt{\frac{\gamma}{\alpha}}.
\end{equation}
 Compare this barrier function to the barrier function used at the start of Lemma 6. 
Applying the discrete maximum principle and using $\theta >1$ it follows that
\begin{equation}\label{sharpVbnd}
|(V-v)(x_i)|\leq  CN^{-1}(x_i-1+\tilde \psi (x_i)).
\end{equation}
In order to use this to find a bound on $D^+(V-v)(1)$ we need to bound $D^{-}\tilde \psi(1)$. Defining $F_i := D^{-} \tilde \psi(x_i)$, using \eqref{eqnpsi} we see that
\[
-\ve\left(\frac{F_{i+1}-F_{i}}{\bar{h}_i}\right)+C_{*}\sqrt{\ve\theta}F_{i} =0, \quad x_i \in (0,1).
\]
Then
\begin{equation}\label{F}
F_{i} = \prod_{k = i}^{N-1}(1+C_*\sqrt{\ve\theta}\ve^{-1}\bar{h}_k)^{-1}F_N, \quad \hbox{for} \quad i<N;
\end{equation}
where the constant $F_N$ is to be determined.
By telescoping, we see that
\[
h_{L}\sum_{i = 1}^\frac{N}{4} F_i+H\sum_{i = \frac{N}{4}+1}^\frac{3N}{4} F_i+h_{R}\sum_{i = \frac{3N}{4}}^N F_i = \tilde \psi (1) -\tilde \psi (0) = -1;
\]
and from \eqref{F} it follows that
\begin{eqnarray*}
|F_{N}| &\leq& \frac{1}{h_R(\sum_{i = \frac{3N}{4}}^{N-1}(1+C_*\sqrt{\ve\theta} \ve^{-1}h_R)^{-(N-i)}+1)} \\
&\leq&
\frac{C_*\sqrt{\ve\theta}\ve^{-1}}{(1+C_*\sqrt{\ve\theta} \ve^{-1}h_R)(1-(1+\sqrt{\ve\theta} \ve^{-1}h_R)^{-\frac{N}{4}}))}.
\end{eqnarray*}
 For $N$ large enough we conclude that
\[
D^{-}\tilde \psi(1) = |F_{N}| \leq C\sqrt{\frac{\theta}{\ve}}.
\]
Using this bound and $(V-v)(1)=0$, we have established the bound
\begin{eqnarray*}
\vert D^-(V-v)(1)\vert &\le&C( N^{-1}(1+\vert D^- \tilde \psi(1)\vert))\leq CN^{-1}(1+\sqrt{\frac{\theta}{\ve}}).
\end{eqnarray*}
This yields the desired bound at $x=1$ in the convection-diffusion case where $\theta >1$.

(iv) For the reaction-diffusion case, where $\theta=1$ the argument is more complicated. Consider
\[
B_4(x_i) := C_1(\sqrt(\ve)N^{-2}\ln N R_1)+C_2(N^{-2}\frac{\sqrt{\ve}}{\beta}+N^{-1}(x_i-1+\tilde \psi(x_i))
\]
with $R_1, \tilde \psi$ are as defined previously in (\ref{wedge}) and (\ref{eqnpsi}) respectively. 
This fourth barrier function is a minor alteration to the barrier function $B_2(x_i)$. We can show
\[
L^{N}(N^{-2}\frac{\sqrt{\ve}}{\beta}+N^{-1}(x_i-1+\psi(x_i)) \geq CN^{-1}(\mu +N^{-1}\sqrt{\ve})
\]
and using \eqref{lnr1} we see
\[
L^{N} B_4(x_i) \geq  \left\{
\begin{array}{lll}
CN^{-1}(\mu +N^{-1}\sqrt{\ve}), & \hbox{if } x_i < \sigma_L,\ \sigma_L<x_i < \sigma_R  \\
CN^{-1}(\sqrt{\ve}N^{-1}\ln N (\frac{\ve N}{\sigma_L})+\mu ), & \hbox{if } x_i = \sigma_L, 1-\sigma_R\\
C_1(\sqrt{\ve}N^{-2}\ln N (-\frac{\mu a}{\sigma_R})+C_2N^{-1}(\mu +N^{-1}\sqrt{\ve}), & \hbox{if } x_i > 1-\sigma_R .
\end{array}
\right.
\]
As before, as $\theta = 1$,
\[\sqrt{\ve}N^{-2}\ln N (\frac{\ve N}{\sigma_L}) = \sqrt{\ve}N^{-2}\ln N (\frac{\ve N}{\sigma_R}) = \frac{N^{-1}\ve\sqrt{\gamma \alpha}}{4};
\]
and also for $C_2$ sufficiently large
\[
C_1(\sqrt{\ve}N^{-2}\ln N (-\frac{\mu a}{\sigma_R})+C_2(\mu N^{-1})\geq C_3\mu N^{-1}.
\]
We therefore have
\[
L^{N} B_4(x_i) \geq  \left\{
\begin{array}{lll}
CN^{-1}(\mu +N^{-1}\sqrt{\ve}), & \hbox{if }\  x_i < \sigma_L, \ \sigma_L<x_i < \sigma_R,\\
CN^{-1}(\mu +{\ve}), & \hbox{if}\ x_i = \sigma_L,1-\sigma_R\\
C_1N^{-1}\sqrt{\ve}N^{-1}\ln N (-\frac{\mu a}{\sigma_R})+C_2N^{-1}(\mu +N^{-1}\sqrt{\ve}), & \hbox{if }\  x_i > 1-\sigma_R .
\end{array}
\right.
\]
Using the discrete maximum principle we deduce that
\[|(V-v)(x_i)| \leq C_1(\sqrt{\ve} N^{-2}\ln N R_1)+C_2(N^{-2}\frac{\sqrt{\ve}}{\beta}+N^{-1}(x_i-1+\tilde \psi(x_i)),
\]
which yields the bound
\[\vert {\cal V}^-_{N}\vert \le C_1(\frac{\sqrt{\ve} N^{-2}\ln N (\frac{h_R}{\sigma_R})}{h_R})+C_2(\frac{N^{-2}\sqrt{\ve}}{\beta h_R}+\frac{N^{-1}h_R}{h_R}+N^{-1}\vert D^- \tilde \psi(1)\vert)).
\]
Simplifying we have
\[
\vert D^-(V-v)(1)\vert \le CN^{-1}(1+\sqrt{\frac{\theta}{\ve}})
\]
and this completes the proof.
\end{proof}

\vskip0.25cm 

{\bf Appendix C. Proof of Theorem 9.}

\begin{proof}
 (i) At the interior points, using the truncation error bounds (\ref{trunc_x}), we can establish that
\[
\vert \hat L^{N} D^-(V-v)(x_i)\vert  \leq \left\{
\begin{array}{lllll}
C N^{-1}, & \hbox{if } x_i\ne \sigma_L+H, \sigma _L, 1-\sigma _R, 1-\sigma_R+h_R,
\\
 C(\frac{\sqrt{\ve}}{\sqrt{\theta}\ln N}+N^{-1}), & \hbox{if } x_i=\sigma _L, \\
 C(\ve+\mu+N^{-1}),  & \hbox{if } x_i=\sigma_L+H, \\
 C(\ve+N^{-1}), & \hbox{if } x_i=1-\sigma _R, \\
 C(\frac{\sqrt{\ve\theta}+ \sqrt{\frac{\theta\mu^2}{\ve}}}{\ln N}+N^{-1}),  & \hbox{if } x_i=1-\sigma _R +h_R,
\end{array} \right.
\]
We next define a combination of barrier functions, which allow us establish a bound on \[
\sqrt{\frac{\ve}{\theta}} \vert D^-(V-v)(x_i)\vert.\] This initial set of barrier functions are linear and step functions. In order to establish  the sharper bounds on  $\vert D^-(V-v)(x_i)\vert $ these barrier functions are replaced by discrete exponential
barrier functions. Define the following ramp functions
\[R_2(x_i) :=  \left\{
\begin{array}{ll}
\frac{x_i}{\sigma_L}, & \hbox{if } x_i \leq \sigma_L, \\
1, & \hbox{if }   \sigma_L<x_i\leq 1
\end{array}
\right.  \quad R_3(x_i) :=  \left\{
\begin{array}{ll}
\frac{x_i}{1-\sigma_R}, & \hbox{if }  x_i \leq  1-\sigma_R.\\
1, & \hbox{if } 1-\sigma_R< x_i  \leq 1\\
\end{array}
\right. \]
and step functions
\[S_1(x_i) :=  \left\{
\begin{array}{ll}
0, & \hbox{if } x_i \leq \sigma_L, \\
1, & \hbox{if }   \sigma_L<x_i\leq 1
\end{array}
\right.  \quad S_2(x_i) :=  \left\{
\begin{array}{ll}
0, & \hbox{if }  x_i \leq  1-\sigma_R.\\
1, & \hbox{if } 1-\sigma_R< x_i  \leq 1\\
\end{array}
\right.  \]
We find that
\begin{eqnarray*}
\hat{L}^{N} R_2(x_i) &\geq&  \left\{
\begin{array}{lll}
\frac{\mu\alpha}{\sigma_L}, & \hbox{if } x_i < \sigma_L , \\
\frac{N\gamma\alpha}{64\theta (\ln N)^2}+\frac{\mu\alpha}{\sigma_L}, & \hbox{if } x_i = \sigma_L\\
b+\mu D^{-}a  , & \hbox{if } x_i > \sigma_L , \\
\end{array}
\right.
\\ 
\hat{L}^{N} R_3(x_i) &\geq&  \left\{
\begin{array}{lll}
\mu\alpha, & \hbox{if } x_i < \sigma_L, \sigma_L<x_i < 1-\sigma_R  \\
-\frac{\ve}{(h_l+H)h_l}\left(\frac{H-h_L}{1-\sigma_R}\right), & \hbox{if } x_i = \sigma_L\\
\frac{\ve N}{2}+\mu\alpha, &\hbox{if } x_i = 1-\sigma_R, \\
b+\mu D^{-1}a, & \hbox{if } x_i > 1-\sigma_R .
\end{array}
\right.
\\
\hat{L}^{N} S_1(x_i) &\geq&  \left\{
\begin{array}{lll}
0, & \hbox{if } x_i < \sigma_L , \\
\frac{-N^2\gamma\alpha}{16\theta(\ln N)^2} , & \hbox{if } x_i = \sigma_L\\
\frac{\ve N^2}{4}+\frac{\mu\alpha}{H}, & \hbox{if } x_i = \sigma_L+H , \\
b+\mu D^{-}a, & \hbox{if } x_i > \sigma_L+H .
\end{array}
\right.
\\
\hat{L}^{N} S_2(x_i) &\geq&  \left\{
\begin{array}{lll}
0, & \hbox{if } x_i < 1-\sigma_R , \\
\frac{-2\ve}{(H+h_r)h_r}, & \hbox{if } x_i = 1-\sigma_R\\
\frac{\sqrt{\ve\theta}N^{2}}{8\ln N}, & \hbox{if } x_i =1 -\sigma_R+h_r , \\
b+\mu D^{-}a, & \hbox{if } x_i > 1-\sigma_R .
\end{array}
\right.
\end{eqnarray*}

Consider the barrier function
\[ 
\begin{split}
B_4(x_i) =N^{-1}\bigl( C_1\sqrt{\ve\theta}\ln N R_2+C_2(1+\frac{\mu}{\ve})(S_1N^{-1}+4R_2)+C_3(R_2+R_3)\\+C_4(1+\frac{\mu}{\ve})(N^{-1}S_2+2(R_2+R_3))+C_5\bigr).\end{split}
\] 
We find that $\hat{L}^N(B_4\pm {\mathscr V}^-_{i}) \geq 0$ and applying the maximum principle we get derivative bounds with scaling everywhere.
That is, we have established the error bound
\[
\sqrt{\frac{\ve}{\theta}}\vert D^-(V-v)(x_i) \vert \leq C N^{-1} \ln N, \quad x_i \in (0,1].
\]
We now proceed to improve on this error bound.

(ii)
Consider first the case of $\theta = 1$. Instead of using barrier functions involving ramps to deal with the truncation error at $x_i= \sigma_L+H$ and $x_i = 1-\sigma_r+h_r$ we define the following two mesh functions
\begin{eqnarray*}
Z_L(x_i) &:=&  \left\{
\begin{array}{ll}
(1+\rho _L H)^{-1}(1+\rho _L h_L)^{i-\frac{N}{4}-1}, & \hbox{if }  x_i \leq \sigma_L, \\
1, & \hbox{if}\  \sigma_L < x_i \leq 1
\end{array}
\right.\\ \\
Z_R(x_i) &:=&  \left\{
\begin{array}{ll}
1, & \hbox{if }  x_i \leq 1-\sigma_R\\
(1+0.5\rho _R h_R)^{\frac{3N}{4}+1-i}, & \hbox{if  }  x_i \geq 1-\sigma_R+h_R, \\
\end{array}
\right.
\end{eqnarray*}\vspace{-0.5cm}

\noindent Remembering that we are in the case where $\theta = 1$, we have
\[
\hat L^{N} Z_L(x_i) \geq  \left\{
\begin{array}{llll}
0, & \hbox{if }  x_i < \sigma_L\\
\frac{-\gamma\alpha N}{16\ln N}, & \hbox{if }  x_i = \sigma_L, \\
\frac{\mu\alpha N}{4} , & \hbox{if } x_i = \sigma_L+H, \\
b, & \hbox{if } x_i > \sigma _L+H.
\end{array}
\right.
\]
and
\[
\hat L^{N} Z_R(x_i) \geq  \left\{
\begin{array}{llll}
b, & \hbox{if }  x_i <1-\sigma_R+h_R\\
\frac{\alpha\gamma N}{128\ln N}, & \hbox{if }  x_i = 1-\sigma_R+h_R, \\
0, & \hbox{if } x_i > 1-\sigma_R+h_r.
\end{array}
\right.
\]
Consider the barrier function
\[ 
\begin{split}B_5(x_i) = N^{-1}(C_1\sqrt{\ve}\ln NR_2 + C_2(Z_L+\ln N R_2)+C_3(N^{-1}S_1+\frac{R_2}{4})\\
+C_4(R_2+R_3)+C_5\left(\sqrt{\ve}+\frac{\mu}{\sqrt{\ve}}\right)Z_R+C_6)\end{split}
\] 
and use the maximum principle to deduce that
\begin{equation}\label{react-diff-smooth}
\vert D^- (V-v)(x_i) \vert \le C N^{-1} (\ln N),\quad  \hbox{if } 0 < x_i \leq 1; \quad \hbox{and} \ \theta= 1.
\end{equation}

(iii) Next consider the case of $\theta >1$. Define the following three mesh functions;
\begin{eqnarray*}
P(x_i) &:=&  \left\{
\begin{array}{llll}
0
, & \hbox{if }  x_i \leq \sigma _L \\
(1+ 0.5\rho_R H)^{i-\frac{3N}{4}}, & \hbox{if } \sigma _L+H\leq x_i \leq 1-\sigma _R, \\
 (1+\frac{\rho _Rh_R}{32})^{i-3N/4-1}, & \hbox{if } 1-\sigma _R+h_R\leq x_i \leq 1.
\end{array}
\right. \\ \\ \\
Q(x_i) &:=&  \left\{
\begin{array}{ll}
0, & \hbox{if }  x_i \leq 1-\sigma _R, \\
(1+\frac{\rho h_R}{32})^{i-3N/4-1}, & \hbox{if }\quad  1-\sigma _R+h_R\leq x_i \leq 1.
\end{array}
\right. \\ \\ \\
\hat{Z}_{L}(x_i) &:=&  \left\{
\begin{array}{ll}
(1+\frac{\rho _R H}{4})^{-1}(1+0.5 \rho_R h_L)^{i-\frac{N}{4}-1}, & \hbox{if }\quad   x_i \leq \sigma_L, \\
1, & \hbox{if} \quad  \sigma_L < x_i \leq 1
\end{array} .
\right. 
\end{eqnarray*}
%
Observe that
\[
\hat L^{N} P(x_i) \geq  \left\{
\begin{array}{llll}
0, & \hbox{if }  x_i < \sigma_L  , \\
-\frac{\ve N}{4\mu \ln N} , & \hbox{if } x_i = \sigma _L, \\
0 & \hbox{if } \sigma_L+H\leq x_i < 1-\sigma _R \\
\frac{\mu\alpha N}{4} & \hbox{if } x_i = 1-\sigma _R \\
-\frac{\mu^2\alpha^2N}{512\ve\ln N} & \hbox{if } x_i = 1-\sigma _R+h_r \\
bP(x_i), & \hbox{if } x_i > 1-\sigma _R+h_R,
\end{array}
\right.
\]
and
\begin{eqnarray*}
\hat L^{N} Q(x_i) \geq  \left\{
\begin{array}{llll}
0, & \hbox{if }  x_i < 1-\sigma _R, \\
-2\ve N^2 , & \hbox{if } x_i = 1-\sigma _R, \\
 \frac{\mu\alpha N^2}{32\ln N}, & \hbox{if } x_i = 1-\sigma _R+h_R, \\
bQ(x_i), & \hbox{if } x_i > 1-\sigma _R+h_R.
\end{array}
\right.
\\
\hat L^{N} \hat{Z}_L(x_i) \geq  \left\{
\begin{array}{llll}
0, & \hbox{if }  x_i \leq \sigma_L, \\
\frac{\mu\alpha N}{8} , & \hbox{if } x_i = \sigma_L+H, \\
0, & \hbox{if } x_i > \sigma _L+H.
\end{array}
\right.
\end{eqnarray*}
Considering the linear combination
\[
\Psi(x_i)=N^{-1} \bigl( \left(1+\frac{\mu}{\ve}\right)N^{-1}Q+\frac{8}{\alpha}\left(\frac{\ve}{\mu}+1\right)P+\frac{128}{\gamma^2\alpha}(\mu+\ve)\ln N R_2\bigr),
\]
we see that
\[
\hat L^{N} (\Psi(x_i)) \geq  \left\{
\begin{array}{lll}
0 , & \hbox{if } x_i \neq 1-\sigma _R+h_R, \\
\frac{C}{\ln N}\left(\mu+\frac{\mu^2}{\ve}\right), & \hbox{if } x_i = 1-\sigma _R+h_R.
\end{array}
\right.
\]
Use the barrier function
\[
B_6(x_i)=N^{-1} \bigl( C_1\mu \ln N R_2+C_2\left(\frac{\ve}{\mu}+1\right)\hat{Z}_L+ C_3(R_2+R_3) + C_4N\Psi+C_5\bigr).
\]
to derive the bound
\[
|(D^- (V-v)(x_i))|\leq N^{-1} \bigl(C_1\mu \ln N R_2+C_2\left(\frac{\ve}{\mu}+1\right)\hat{Z}_L+ C_3(R_2+R_3) + C_4N\Psi+C_5\bigr).
\]
If $x_i\leq 1-\sigma_R$ we  have established the bound $
|(D^- (V-v)(x_i))|\leq CN^{-1};
$
and we  have removed all  scaling outside the computational layer region on the right.
\end{proof}

\end{document}